\documentclass[a4paper, reqno, 11pt]{amsart}
\usepackage{color}
\makeatletter

\@addtoreset{equation}{section}
\makeatother
\usepackage{setspace}
\usepackage{xcolor}
\usepackage{here}
\usepackage{graphicx}
\usepackage{float}
\usepackage{hyperref}
\usepackage[T1]{fontenc}
\everymath{\displaystyle}
\usepackage{amsmath}
\usepackage{amssymb}
\usepackage{latexsym}
\usepackage{amsthm}
\usepackage{hyperref}
\usepackage{mathtools}
\usepackage{mathrsfs}
\newtheorem{Thm}{Theorem}[section]

\newtheorem{Lem}[Thm]{Lemma}
\newtheorem{Prop}[Thm]{Proposition}

\newtheorem{Cor}[Thm]{Corollary}

\theoremstyle{definition}

\newtheorem{Def}[Thm]{Definition}
\newtheorem{Rem}[Thm]{Remark}
\newtheorem{Exa}[Thm]{Example}

\begin{document}

\begin{abstract}
We establish a cutoff criterion for finite Markov chains with non-negative Bakry--\'Emery curvature without assuming that the support of the transition matrix is symmetric.
This extends a criterion of Salez based on a refined product condition. 
For each allowed transition, we consider a shortest directed path from its endpoint back to its starting point. 
Our cutoff criterion depends on the maximum length of these paths. 
We also obtain a discrete-time counterpart for lazy chains.
For random walks on finite Abelian groups, we prove non-negative Bakry--\'Emery curvature for arbitrary generating sets.
As an application, we establish almost-sure cutoff for random walks on products of cyclic groups of order three.
\end{abstract}

\title[]{Cutoff on non-negatively curved chains without symmetric support condition}
\author{Kazuki Okamura}
\date{\today}
\address{Department of Mathematics, Faculty of Science, Shizuoka University, 836, Ohya, Suruga-ku, Shizuoka, 422-8529, JAPAN.}
\email{okamura.kazuki@shizuoka.ac.jp}

\maketitle

\section{Introduction}

The cutoff phenomenon is a sharp transition in the convergence of a Markov chain to its stationary distribution. 
The total variation distance to the stationary distribution stays close to its maximal value for a long time, and then drops to zero on a much shorter time scale.
It was discovered by Aldous, Diaconis and Shahshahani in the context of card shuffling \cite{DS1981, Aldous1983, AD1986},
and it has been established for a large number of specific chains; see \cite{Diaconis1996, SaloffCoste2004, LP2017}.
However, the proofs usually require a detailed knowledge of the chain under consideration.
It is a fundamental problem to find general conditions which ensure the occurrence of cutoff without determining the precise location of the cutoff time.
Let $t_{\mathrm{rel}}$ and $t_{\mathrm{mix}}(\varepsilon)$ denote the relaxation time and the mixing time at precision $\varepsilon \in (0,1)$, respectively.
For reversible chains, cutoff implies the product condition $t_{\mathrm{rel}} / t_{\mathrm{mix}}(\varepsilon) \to 0$ for every fixed $\varepsilon \in (0,1)$. 
Peres \cite{Peres2004} conjectured that the product condition is also sufficient for cutoff.
This is true for birth-and-death chains \cite{DLP2010} and for random walks on trees \cite{BHP2017}, but it is false in general; see \cite[Section 6]{CSC2008} and \cite[Chapter 18]{LP2017}.

Salez \cite{Salez2024} showed that, for chains with non-negative curvature in the sense of Bakry--\'Emery \cite{BE1985} or in the sense of Ollivier \cite{Ollivier2009}, cutoff occurs under a suitably refined version of the product condition.
More precisely, let $(\mathcal{X}_{(n)}, P_{(n)})$, $n \ge 1$, be a sequence of finite irreducible Markov chains with non-negative curvature such that the support of $P_{(n)}$ is symmetric, that is, $P_{(n)}(x,y) > 0$ if and only if $P_{(n)}(y,x) > 0$.
Let $\Delta_n$ be the inverse of the smallest non-zero off-diagonal entry of $P_{(n)}$,
and let $t_{\mathrm{rel}}^{(n)}$ and $t_{\mathrm{mix}}^{(n)}(\cdot)$ be the relaxation time and the mixing time of the continuous-time chain associated with $(\mathcal{X}_{(n)}, P_{(n)})$. 
\cite[Theorem 1]{Salez2024} states that a cutoff occurs if for every $\varepsilon \in (0,1)$,
\[ \lim_{n \to \infty} \frac{\left(t_{\mathrm{rel}}^{(n)} \log \Delta_n\right)^{2}}{t_{\mathrm{mix}}^{(n)}(\varepsilon)} = 0. \]
This condition involves only the orders of magnitude of $t_{\mathrm{rel}}^{(n)}$, $t_{\mathrm{mix}}^{(n)}(\varepsilon)$ and $\Delta_n$,
and it is satisfied by random walks on Abelian Cayley graphs with reasonably good expansion, in particular by random walks on almost all Abelian Cayley graphs with sufficiently many generators \cite[Corollaries 3 and 4]{Salez2024}.

The proof in \cite{Salez2024} consists of two parts.
The first part is an entropic criterion for cutoff, which is formulated in terms of the varentropy of the heat kernel and holds for every stochastic matrix.
The second part is an upper bound for the varentropy of non-negatively curved chains, which is obtained by combining a local concentration inequality for Lipschitz functions with a gradient estimate for the logarithm of the heat kernel.
Reversibility is not assumed in \cite{Salez2024}.
However, the symmetry of the support is used in several places in the second part. 
It makes the graph distance on $\mathcal{X}_{(n)}$ symmetric, which is used in the definition of the Ollivier--Ricci curvature and in the diameter bound;
it is used in the gradient estimate, both for the bound $\Delta(P^{*}) \le \Delta(P)^{2}$ for the adjoint $P^{*}$ of $P$ with respect to the stationary distribution and for controlling the logarithm of the heat kernel in both directions along each transition. 

The purpose of this paper is to remove the symmetric support condition from \cite[Theorem 1]{Salez2024}.
The key quantity in our approach is $k_{(\mathcal{X}, P)}$ introduced in Definition \ref{def:k} below.
For $x, y \in \mathcal{X}$, let $d_{P}(x \to y)$ be the minimal number of steps of the chain from $x$ to $y$.
This is not symmetric in $x$ and $y$ in general.
The quantity $k_{(\mathcal{X}, P)}$ is the maximum of $d_{P}(y \to x)$ over all pairs $x \ne y$ with $P(x,y) > 0$,
that is, the maximal number of steps needed to go back along a transition of the chain.
It holds that $k_{(\mathcal{X}, P)} \ge 1$, with equality if and only if the support of $P$ is symmetric.
Our main result, Theorem \ref{thm:main}, states that a sequence of chains with non-negative Bakry--\'Emery curvature exhibits cutoff under the condition \eqref{eq:main-assumption},
which is obtained from the above condition by replacing $\log \Delta_n$ with $k_n^{2}\left(1 + \log \Delta_n\right)$, where $k_n$ is the quantity $k_{(\mathcal{X}, P)}$ of the $n$-th chain;
see the paragraph after Theorem \ref{thm:main} for the precise relation between the two conditions.
If the supports are symmetric, then $k_n = 1$, and Theorem \ref{thm:main} contains the cutoff assertion of \cite[Theorem 1]{Salez2024} in the case of the Bakry--\'Emery curvature.
We deal only with the Bakry--\'Emery curvature.
The Ollivier--Ricci curvature considered in \cite{Salez2024} is defined in terms of the graph distance, which is not symmetric without the symmetric support condition, and we do not deal with it in this paper.

The proof of Theorem \ref{thm:main} follows the strategy of \cite{Salez2024}.
The entropic criterion \cite[Theorem 5]{Salez2024} is available without any assumption on the support.
We establish the gradient estimate for the logarithm of the heat kernel (Lemma \ref{lem:log-Pt-Lip} and Proposition \ref{thm:3.14}) and the diameter bound (Proposition \ref{prop:6.3}) for chains whose support is not necessarily symmetric.
In these estimates, going back along a transition of the chain is replaced by going along a path of length at most $k_{(\mathcal{X}, P)}$, and this is how the quantity $k_{(\mathcal{X}, P)}$ enters the estimates.
The bound $t_{\mathrm{rel}} \ge 1/2$ holds without the symmetric support condition (Lemma \ref{lem:trel-lower}). 
We also give a sufficient condition for \eqref{eq:main-assumption} in terms of the sizes $|\mathcal{X}_{(n)}|$ (Proposition \ref{prop:sufficient}), which corresponds to \cite[Corollary 2]{Salez2024},
and the discrete-time counterpart of Theorem \ref{thm:main} for lazy chains (Theorem \ref{thm:discrete}) by using the comparison theorem of Chen and Saloff-Coste \cite{CSC2013}.

Non-symmetric supports appear naturally for random walks on groups.
Let $G$ be a finite Abelian group and let $S$ be a generating set of $G$.
The random walk on $G$ which moves from $x$ to $xs$ with probability $1/|S|$ for each $s \in S$ has symmetric support if and only if $S = S^{-1}$.
Random walks on finite groups driven by random generating sets which are not symmetrized, called random random walks, have been studied by Hildebrand \cite{Hildebrand1994} and Dou and Hildebrand \cite{DH1996}; see \cite{Hildebrand2005} for a survey.
In Section \ref{sec:example}, we show that the random walk on a finite Abelian group driven by an arbitrary generating set has non-negative Bakry--\'Emery curvature (Theorem \ref{thm:abelian}).
We then consider the groups $(\mathbb{Z}/3\mathbb{Z})^{d}$, $d \ge 1$. 
We show that if $S_{d}$ is a uniformly distributed random subset of $(\mathbb{Z}/3\mathbb{Z})^{d} \setminus \{0\}$ of size $\lceil cd \rceil$ with $c > 2\log 3$,
then, almost surely, the random walks driven by $S_{d}$ have non-symmetric supports for every sufficiently large $d$ and exhibit cutoff (Theorem \ref{thm:z3d-random}).
This is a counterpart of \cite[Corollary 4]{Salez2024}, in which the generating sets are symmetrized, and the probability that the assumptions fail for a given $d$ is bounded explicitly.

\subsection{Framework and main result}\label{subsec:framework}

Let $(\mathcal{X}, P)$ be a finite irreducible Markov chain with $|\mathcal{X}| \ge 2$. 
Denote the stationary distribution by $\pi$. 
Consider the continuous-time Markov chain $(X_t)_t$ associated with $(\mathcal{X}, P)$ whose transition probability satisfies 
\[\mathscr{P}_t (x,y) \coloneqq  e^{-t} \sum_{n=0}^{\infty} \frac{t^n}{n!} \, P^n (x,y), \qquad x,y \in \mathcal{X},\ t \ge 0, \]
where $P^n$ denotes the $n$-step transition probabilities of $P$.

For two probability measures $\mu_1,\mu_2$ on $\mathcal{X}$, the total variation distance is defined by
\[ \|\mu_1 - \mu_2\|_{\mathrm{TV}} \coloneqq  \frac{1}{2} \sum_{x \in \mathcal{X}} |\mu_1(x) - \mu_2(x)|. \] 

Define the mixing time by 
\[ t_{\mathrm{mix}}(\varepsilon) \coloneqq  \inf\left\{t\ge 0 \,\middle|\, \max_{x\in\mathcal{X}} \left\|\mathscr{P}_t(x,\cdot)-\pi\right\|_{\mathrm{TV}} \le \varepsilon\right\}, \ 0 < \varepsilon < 1.\]

For a function $f : \mathcal{X} \to \mathbb{R}$, define
\[ \mathbb{E}_\pi [f] \coloneqq  \sum_{x \in \mathcal{X}} f(x)\,\pi(x), \]
and
\begin{equation}\label{eq:def-var} 
\mathrm{Var}_\pi (f) \coloneqq  \sum_{x \in \mathcal{X}} \left(f(x) - \mathbb{E}_\pi[f]\right)^2 \,\pi(x). 
\end{equation} 

For functions $f,g : \mathcal{X} \to \mathbb{R}$ and $x \in \mathcal{X}$, define the carr\'e du champ by 
\[ \Gamma(f,g)(x) \coloneqq  \frac{1}{2}  \sum_{y \in \mathcal{X}} \left(f(x)-f(y)\right)\left(g(x)-g(y)\right)\,P(x,y). \]
Let $\Gamma(f)(x) \coloneqq  \Gamma(f,f)(x)$. 
For functions $f, g : \mathcal{X} \to \mathbb{R}$, define the Dirichlet form by 
\[ \mathcal{E}(f,g) \coloneqq  \sum_{x \in \mathcal{X}} \Gamma(f,g)(x)\,\pi(x). \]
We write $\mathcal{E}(f) \coloneqq  \mathcal{E}(f,f)$.

The relaxation time is defined by
\begin{equation}\label{eq:def-relax} 
t_{\mathrm{rel}} \coloneqq  \sup \left\{ \frac{\mathrm{Var}_\pi(f)}{\mathcal{E}(f)} \middle| f : \mathcal{X} \to \mathbb{R},\ \mathcal{E}(f) > 0 \right\}. 
\end{equation}

For a stochastic matrix $Q$ on $\mathcal{X}$, let 
\[ \Delta(Q) \coloneqq \max\left\{\frac{1}{Q(x,y)} \, \middle| \,  Q(x,y) > 0, \ x \ne y \right\}. \]

Set 
\[ d_{P} (x \to y) \coloneqq \min\left\{k \in \mathbb{N}_0 \,\middle|\, P^k(x,y) > 0 \right\}, \  x, y \in \mathcal{X}, \]
where $\mathbb{N}_0 \coloneqq \{0, 1, 2, \dots\}$.
This is finite since $(\mathcal{X},P)$ is irreducible.
We remark that $d_{P} (x \to x) = 0$ and it can happen that $d_{P} (x \to y) \ne d_{P} (y \to x)$. 

\begin{Def}\label{def:k}
Let
\[ k_{(\mathcal{X}, P)} \coloneqq \max\left\{d_{P}(y \to x)\,\middle|\, d_{P}(x \to y) = 1\right\}. \]
\end{Def}

This is finite by irreducibility of the Markov chain.
Since $d_P (x\to x)=0$, if $d_{P}(x\to y)=1$, then $x\ne y$, and hence,
$d_{P} (y\to x)\ge 1$, that is, $k_{(\mathcal{X}, P)}\ge 1$.
We say that the support of $P$ is symmetric if, for every $x, y \in \mathcal{X}$, $P(x,y) > 0$ if and only if $P(y,x) > 0$.
By the definition of $k_{(\mathcal{X}, P)}$, the support of $P$ is symmetric if and only if $k_{(\mathcal{X}, P)} = 1$.

Define the Laplacian by 
\[\mathcal{L} f(x)\coloneqq \sum_{y\in\mathcal{X}} P(x,y)\left(f(y)-f(x)\right), \quad x \in \mathcal{X},\ f \colon \mathcal{X} \to \mathbb{R}.\]

The Bakry--\'Emery curvature is defined by the supremum of $\kappa$ such that  for every $f : \mathcal{X}\to\mathbb{R}$, and for every $x \in \mathcal{X}$, 
\[ \frac{1}{2} \mathcal{L}\Gamma(f)(x)\ge \Gamma(f,\mathcal{L}f)(x)+\kappa\,\Gamma(f)(x).\] 
We denote this quantity by $\textup{BE}(\mathcal{X}, P)$.

Let $\mathcal{X}_{(n)} = \left(\mathcal{X}_{(n)},\,P_{(n)},\pi_{(n)}\right), n \ge 1$, be a sequence of finite irreducible Markov chains.
For the $n$-th chain $\mathcal{X}_{(n)}$,
we denote by $(\mathscr{P}_t^{(n)})_{t \ge 0}$ the continuous-time semigroup associated with $(\mathcal{X}_{(n)}, P_{(n)})$,
and by $t_{\mathrm{mix}}^{(n)}(\cdot)$ and $t_{\mathrm{rel}}^{(n)}$ the mixing time and the relaxation time of $(\mathcal{X}_{(n)}, P_{(n)})$, respectively.
Let $k_n \coloneqq k_{(\mathcal{X}_{(n)}, P_{(n)})}$ and $\Delta_n \coloneqq \Delta(P_{(n)})$.

We say that a cutoff occurs for $(\mathcal{X}_{(n)})_n$ if for every $\varepsilon \in (0,1)$,
\[ \lim_{n\to\infty} \frac{t_{\mathrm{mix}}^{(n)}(\varepsilon)}{t_{\mathrm{mix}}^{(n)}(1-\varepsilon)} = 1. \]

\begin{Thm}\label{thm:main}
Let $\mathcal{X}_{(n)} = \left(\mathcal{X}_{(n)},\,P_{(n)},\pi_{(n)}\right), n \ge 1$,
be a sequence of finite irreducible Markov chains with non-negative Bakry--\'Emery curvature.
Assume that for every $\varepsilon \in \left(0,1/2\right)$,
\begin{equation}\label{eq:main-assumption}
\lim_{n \to \infty} \frac{k_n^{2}\left(1+\log \Delta_n\right)\,t_{\mathrm{rel}}^{(n)}}{\sqrt{t_{\mathrm{mix}}^{(n)}(1-\varepsilon)}} = 0.
\end{equation}
Then a cutoff occurs for $\left(\mathcal{X}_{(n)} \right)_n$.
\end{Thm}

We compare Theorem \ref{thm:main} with \cite[Theorem 1]{Salez2024}.
Since $t_{\mathrm{mix}}^{(n)}(\cdot)$ is non-increasing, \eqref{eq:main-assumption} is equivalent to the condition that for every $\varepsilon \in (0,1)$,
\[ \lim_{n \to \infty} \frac{\left(k_n^{2}\left(1+\log \Delta_n\right) t_{\mathrm{rel}}^{(n)}\right)^{2}}{t_{\mathrm{mix}}^{(n)}(\varepsilon)} = 0. \]
Assume that the support of $P_{(n)}$ is symmetric and $|\mathcal{X}_{(n)}| \ge 3$ for every $n \ge 1$, as in \cite{Salez2024}.
Then $k_n = 1$, and $\Delta_n \ge 2$ by Remark \ref{rem:delta-symmetric} below, so that $\log \Delta_n \le 1 + \log \Delta_n \le \left(1 + 1/\log 2\right) \log \Delta_n$.
Hence, in this case, \eqref{eq:main-assumption} is equivalent to the condition that for every $\varepsilon \in (0,1)$,
\[ \lim_{n \to \infty} \frac{\left(t_{\mathrm{rel}}^{(n)} \log \Delta_n\right)^{2}}{t_{\mathrm{mix}}^{(n)}(\varepsilon)} = 0, \]
which is the assumption of \cite[Theorem 1]{Salez2024}.
Therefore Theorem \ref{thm:main} contains the cutoff assertion of \cite[Theorem 1]{Salez2024} in the case of the Bakry--\'Emery curvature.
The factor $k_n^{2}$ in \eqref{eq:main-assumption} comes from the gradient estimate for the logarithm of the heat kernel in Proposition \ref{thm:3.14}.
A quantitative bound for the cutoff window in terms of $k_{(\mathcal{X}, P)}$, $\Delta(P)$, $t_{\mathrm{rel}}$ and $t_{\mathrm{mix}}(\cdot)$ is given in Proposition \ref{thm:5.1}.

We establish the discrete time counterpart of this assertion.
We say that $(\mathcal{X}, P)$ is $\delta$-lazy for $\delta \in (0,1)$ if $P(x,x) \ge \delta$ for every $x \in \mathcal{X}$. 
Define the discrete mixing time by 
\[ t_{\mathrm{mix}, d}(\varepsilon) \coloneqq  \inf\left\{n \ge 0 \,\middle|\, \max_{x\in\mathcal{X}} \left\|P^n (x,\cdot)-\pi\right\|_{\mathrm{TV}} \le \varepsilon\right\}, \ 0 < \varepsilon < 1.\]
For the $n$-th chain $\mathcal{X}_{(n)}$ of a sequence of finite irreducible Markov chains as above,
we denote by $t_{\mathrm{mix},d}^{(n)}(\cdot)$ the discrete mixing time of $(\mathcal{X}_{(n)}, P_{(n)})$.

By Chen and Saloff-Coste \cite[Theorem 3.1]{CSC2013},
we have the following:
\begin{Thm}\label{thm:discrete}
Let $\delta \in (0,1)$.
Let $\mathcal{X}_{(n)} = \left(\mathcal{X}_{(n)},\,P_{(n)},\pi_{(n)}\right), n \ge 1$, be a sequence of finite irreducible $\delta$-lazy Markov chains with non-negative Bakry--\'Emery curvature.
Assume that \eqref{eq:main-assumption} holds for every $\varepsilon \in \left(0,1/2\right)$.
Then a cutoff occurs for the discrete mixing times of $\left(\mathcal{X}_{(n)} \right)_n$, that is, for every $\varepsilon \in (0,1)$,
\[ \lim_{n\to\infty} \frac{t_{\mathrm{mix},d}^{(n)}(\varepsilon)}{t_{\mathrm{mix},d}^{(n)}(1-\varepsilon)} = 1. \]
\end{Thm}

The rest of this paper is organized as follows.
In Section \ref{subsec:varentropy}, we recall the entropic criterion of \cite{Salez2024}, establish the gradient estimate for the logarithm of the heat kernel for chains whose support is not necessarily symmetric, and derive a bound for the cutoff window.
In Section \ref{subsec:diameter}, we give the diameter bound.
In Section \ref{subsec:cutoff}, we complete the proof of Theorem \ref{thm:main} and give a sufficient condition for \eqref{eq:main-assumption} in terms of $|\mathcal{X}_{(n)}|$.
In Section \ref{subsec:discrete}, we prove Theorem \ref{thm:discrete}.
Section \ref{sec:example} is devoted to the examples on finite Abelian groups.

\section{Proof}\label{sec:proof}

\subsection{Varentropy and the log-gradient estimate}\label{subsec:varentropy}

For distributions $\mu,\pi$  on $\mathcal{X}$, let 
\[ d_{\mathrm{KL}}(\mu\|\pi) \coloneqq  \sum_{x\in\mathcal{X}} \mu(x)\log\!\left(\frac{\mu(x)}{\pi(x)}\right), \]
where $0\log 0 \coloneqq  0$. 
Let 
\[ V_{\mathrm{KL}}(\mu\|\pi) \coloneqq  \sum_{x\in\mathcal{X}} \mu(x) \left(\log\!\left(\frac{\mu(x)}{\pi(x)}\right) - d_{\mathrm{KL}}(\mu\|\pi) \right)^{2}. \]

Furthermore, we let 
\[ d_{\mathrm{KL}}^{*}(t) \coloneqq  \max_{o\in\mathcal{X}} d_{\mathrm{KL}}\!\left(\mathscr{P}_t(o,\cdot)\|\pi\right),  \] 
and 
\[ V_{\mathrm{KL}}^{*}(t)\coloneqq  \max_{o\in\mathcal{X}} V_{\mathrm{KL}}\!\left(\mathscr{P}_{t}(o,\cdot)\|\pi\right). \]
For the $n$-th chain $\mathcal{X}_{(n)}$ of a sequence of finite irreducible Markov chains as in Section \ref{subsec:framework},
we denote by $V_{\mathrm{KL},n}^{*}(\cdot)$ the quantity $V_{\mathrm{KL}}^{*}(\cdot)$ defined for $(\mathcal{X}_{(n)}, P_{(n)})$.

\cite[Theorem 5]{Salez2024} provides that
\begin{Thm}\label{thm:Thm5Salez}
For every $\varepsilon \in (0,\tfrac{1}{2})$,
\[ t_{\mathrm{mix}}(\varepsilon) - t_{\mathrm{mix}}(1-\varepsilon) \le
\frac{2}{\varepsilon^{2}}\,t_{\mathrm{rel}}
\left(1+\sqrt{V_{\mathrm{KL}}^{*}\!\left(t_{\mathrm{mix}}(1-\varepsilon)\right)}\right). \]
\end{Thm}

Now we give an upper bound for $V_{\mathrm{KL}}^{*} (t)$. 
For $f \colon \mathcal{X}\to\mathbb{R}$, let 
\[ \|f\|_{\mathrm{Lip}} \coloneqq \max\left\{|f(u)-f(v)| \,\colon\, d_P (u \to v) \le 1\right\} \]
and $\|f \|_{\infty} \coloneqq \max_{x \in \mathcal{X}} |f(x)|$.

The following corresponds to \cite[Lemma 9]{Salez2024}. 

\begin{Lem}[local concentration inequality]\label{lem:semigroup}
If the Bakry--\'Emery curvature is non-negative, then 
\[ \left\|\mathscr{P}_{t}(f^2)-(\mathscr{P}_{t} f)^2\right\|_{\infty} \le t\,\|f\|_{\mathrm{Lip}}^2\]
for every $t \ge 0$ and every  $f \colon \mathcal{X} \to \mathbb{R}$.
\end{Lem}

\begin{proof}
By the assumption, we see that $\Gamma(\mathcal P_s f)\le \mathcal P_s\Gamma(f)$. 
Therefore, 
\begin{align*}
\mathcal P_t(f^2)-(\mathcal P_t f)^2 = 2\int_0^t \mathcal P_{t-s}\Gamma(\mathcal P_s f)\,ds \le 2t\,\mathcal P_t\Gamma(f) \le t\|f\|_{\mathrm{Lip}}^2.
\end{align*}
\end{proof}

For a stochastic matrix $Q$ on $\mathcal{X}$, let  
\[ \widetilde{\Delta}(Q) \coloneqq \max\left\{\frac{1}{Q(x,y)} \colon Q(x,y)>0\right\}. \]
Then $\widetilde{\Delta}(Q) \ge \Delta(Q)$.

Since $(\mathcal{X}, P)$ is irreducible, $\pi(x) > 0$ for every $x \in \mathcal{X}$.
Let $P^{*}$ be the adjoint of $P$ with respect to $\pi$, that is,
\[ P^{*}(x,y) \coloneqq \frac{\pi(y)}{\pi(x)}\,P(y,x), \quad x, y \in \mathcal{X}. \]
Then $P^{*}$ is a stochastic matrix on $\mathcal{X}$ whose stationary distribution is $\pi$,
and $P^{*}(x,y) > 0$ if and only if $P(y,x) > 0$.

We give a preliminary logarithmic gradient estimate. 

\begin{Lem}\label{lem:log-Pt-Lip}
For $t > 0$ and $o \in \mathcal{X}$,
\[ \left\|\log \frac{\mathscr{P}_t (o, \cdot)}{\pi(\cdot)}\right\|_{\mathrm{Lip}}
\le k_{(\mathcal{X}, P)}\left( \log \Delta(P^*)  + \log\left(e-1+\frac{\log \widetilde{\Delta}(\mathscr{P}_t)}{t}\right) \right). \]
\end{Lem}

We denote $z \to w$ if $d_P (z \to w) = 1$. 
We remark that if $z \to w$, then $z \ne w$. 

\begin{proof} 
For ease of notation, let
\[ f(y) \coloneqq \frac{\mathscr{P}_{t} (o,y)}{\pi(y)}. \]
This is positive for every $y \in \mathcal{X}$ and $t > 0$.

In the same manner as in the proof of \cite[Lemma 10]{Salez2024}, 
we see that 
\[ \sum_{y} P^*(x,y)\,\frac{f(y)}{f(x)} \le e-1+\frac{1}{t}\log\left(\frac{1}{\mathscr{P}_{t} (o,x)}\right). \]

We remark that it can happen that $o=x$; this is the reason why $\widetilde{\Delta}(\mathscr{P}_t)$, rather than $\Delta(\mathscr{P}_t)$, appears in the estimate.
Then we see that
\[ \max_{y:\,y\to x}\frac{f(y)}{f(x)} \le \Delta(P^*)\left(e-1+\frac{1}{t}\log \widetilde{\Delta}(\mathscr{P}_t)\right), \]
that is, 
\[ \max_{y:\,y\to x}\left(\log f(y)-\log f(x)\right) \le \log\Delta(P^*)+\log\left(e-1+\frac{1}{t}\log \widetilde{\Delta}(\mathscr{P}_t)\right).\]

If $y \to x$, then $d_P (x \to y) \le k_{(\mathcal{X}, P)}$ by the definition of $k_{(\mathcal{X}, P)}$,
and hence there exist $\ell \le k_{(\mathcal{X}, P)}$ and $x_1,\dots, x_\ell \in \mathcal{X}$
such that $x \to x_1\to \cdots \to x_\ell = y$.
Hence,
\[ \max_{y:\,y\to x}\left(\log f(x)-\log f(y)\right) \le k_{(\mathcal{X}, P)}\left(\log\Delta(P^*)+\log\left(e-1+\frac{1}{t}\log \widetilde{\Delta}(\mathscr{P}_t)\right) \right). \qedhere \]
\end{proof} 

We give an upper bound for $\Delta(P^*)$. 

\begin{Lem}\label{lem:delta-star}
$\Delta(P^*) \le \Delta(P)^{k_{(\mathcal{X}, P)}+1}$.
\end{Lem}

\begin{proof} 
If $P^*(x,y) > 0$ and $x \ne y$, then $y\to x$ and hence there exist
\(x = x_0,x_1,\dots,x_\ell = y, \ \ell \le k_{(\mathcal{X}, P)}, \)
such that \(x_i \to x_{i+1}\) for every  \(i \in \{0,\dots,\ell-1\} \). 
Assume this is a shortest path, that is, $\ell = d_P (x \to y)$.

By induction on $\ell$, 
we see that for all $\ell \in \mathbb{N}$,
\[ (P^*)^\ell(x,y)=\frac{\pi(y)}{\pi(x)}\,P^\ell(y,x), \ x, y \in \mathcal{X}. \]
Hence, 
\[ P^*(x,y)(P^*)^\ell(y,x) = \frac{\pi(y)}{\pi(x)}P(y,x)\frac{\pi(x)}{\pi(y)}P^\ell(x,y) = P(y,x)\,P^\ell(x,y). \]

Hence, 
\[ \frac{1}{P^*(x,y)} \le \frac{1}{P^*(x,y)(P^*)^\ell(y,x)} = \frac{1}{P(y,x)\,P^\ell(x,y)} \]
\[ \le \frac{1}{P(y,x)}\frac{1}{P(x_0,x_1)} \cdots \frac{1}{P(x_{\ell-1},y)} \le \Delta(P)^{\ell+1}.\]
Since $\Delta(P)\ge 1$ and $\ell\le k_{(\mathcal{X}, P)}$,
\[ \frac{1}{P^*(x,y)} \le \Delta(P)^{k_{(\mathcal{X}, P)}+1}. \qedhere \]
\end{proof}

Let
\[ \mathrm{diam}(\mathcal{X}) \coloneqq \max\left\{d_{P}(x \to y) \,\middle|\, x, y \in \mathcal{X}\right\}, \]
which is finite since $(\mathcal{X}, P)$ is irreducible.
For the $n$-th chain $\mathcal{X}_{(n)}$ of a sequence of finite irreducible Markov chains as in Section \ref{subsec:framework},
we denote by $\mathrm{diam}(\mathcal{X}_{(n)})$ the quantity $\mathrm{diam}(\mathcal{X})$ defined for $(\mathcal{X}_{(n)}, P_{(n)})$.

We give an upper bound for $\widetilde{\Delta}(\mathscr{P}_t)$.
Here and in what follows, we denote by $\mathbb{P}$ and $\mathbb{E}$ the probability and the expectation on a probability space on which the random variables under consideration are defined.
We first prepare an estimate for the Poisson distribution.
It is the inequality $m(D) < D$ for the median $m(D)$ of the gamma distribution with parameter $D$, which is due to Chen and Rubin \cite{ChenRubin1986}; the proof below follows \cite{ChenRubin1986, BergPedersen2006}. 

\begin{Lem}\label{lem:poisson-median}
Let $D \in \mathbb{N}$ and $\lambda \ge D$, and let $N$ be a Poisson random variable with mean $\lambda$.
Then $\mathbb{P}(N \ge D) > 1/2$.
\end{Lem}

\begin{proof}
For $\mu \ge 0$, let $g(\mu) \coloneqq 1 - \sum_{j=0}^{D-1} e^{-\mu} \frac{\mu^{j}}{j!}$, so that $\mathbb{P}(N \ge D) = g(\lambda)$.
Then $g(0) = 0$ and $g'(\mu) = e^{-\mu} \frac{\mu^{D-1}}{(D-1)!}$, and hence, by $\lambda \ge D$,
\[ \mathbb{P}(N \ge D) = \frac{1}{(D-1)!} \int_{0}^{\lambda} e^{-s} s^{D-1}\, ds
 \ge \frac{1}{(D-1)!} \int_{0}^{D} e^{-s} s^{D-1}\, ds. \]
By the substitution $s = D e^{-u}$,
\[ \int_{0}^{D} e^{-s} s^{D-1}\, ds = D^{D} \int_{0}^{\infty} e^{-D(e^{-u} + u)}\, du \]
and
\[ \int_{D}^{\infty} e^{-s} s^{D-1}\, ds = D^{D} \int_{-\infty}^{0} e^{-D(e^{-u} + u)}\, du = D^{D} \int_{0}^{\infty} e^{-D(e^{u} - u)}\, du. \]
Since $e^{u} - u > e^{-u} + u$ for $u > 0$, the former is larger than the latter, and hence
\[ \int_{0}^{D} e^{-s} s^{D-1}\, ds > \frac{1}{2} \int_{0}^{\infty} e^{-s} s^{D-1}\, ds = \frac{(D-1)!}{2}. \qedhere \]
\end{proof}

The following corresponds to  \cite[Eq. (16)]{Salez2024}. 

\begin{Lem}\label{lem:Pt-upper}
Let $\theta \in (0, 1/2]$ and $t \ge \theta\,\mathrm{diam}(\mathcal{X})$. Then
\[ \widetilde{\Delta}(\mathscr{P}_t) \le 2 \left(\frac{\Delta(P)}{\theta}\right)^{\mathrm{diam}(\mathcal{X})}. \]
\end{Lem}

\begin{proof}
Let $D \coloneqq \mathrm{diam}(\mathcal{X})$.
If $D = 0$, then $|\mathcal{X}| = 1$ and $\widetilde{\Delta}(\mathscr{P}_t) = 1$, and the assertion holds.
Assume that $D \ge 1$.
Let $\widehat{P} \coloneqq (1-\theta) I + \theta P$.
Then $\widehat{P}$ is a stochastic matrix on $\mathcal{X}$.
Since $\widehat{P} - I = \theta (P - I)$ and $\widehat{P}$ commutes with $I$,
\[ \mathscr{P}_t = e^{t(P-I)} = e^{(t/\theta)(\widehat{P} - I)} = \sum_{j=0}^{\infty} e^{-t/\theta} \frac{(t/\theta)^{j}}{j!}\, \widehat{P}^{\,j}. \]
Let $x, y \in \mathcal{X}$ and $\ell \coloneqq d_{P}(x \to y) \le D$.
There exist $x = x_0, x_1, \dots, x_\ell = y$ such that $x_i \to x_{i+1}$ for every $i \in \{0, \dots, \ell-1\}$,
and hence $\widehat{P}(x_i, x_{i+1}) = \theta P(x_i, x_{i+1}) \ge \theta/\Delta(P)$.
Since $\widehat{P}(y,y) \ge 1 - \theta \ge \theta/\Delta(P)$ by $\theta \le 1/2$ and $\Delta(P) \ge 1$,
\[ \widehat{P}^{\,D}(x,y) \ge \prod_{i=0}^{\ell-1} \widehat{P}(x_i, x_{i+1}) \cdot \widehat{P}(y,y)^{D-\ell} \ge \left(\frac{\theta}{\Delta(P)}\right)^{D}. \]
For $j \ge D$, since $\widehat{P}^{\,j-D}$ is a stochastic matrix,
\[ \widehat{P}^{\,j}(x,y) = \sum_{z \in \mathcal{X}} \widehat{P}^{\,j-D}(x,z)\, \widehat{P}^{\,D}(z,y) \ge \left(\frac{\theta}{\Delta(P)}\right)^{D}. \]
Let $N$ be a Poisson random variable with mean $t/\theta$.
Then
\[ \mathscr{P}_t(x,y) \ge \sum_{j \ge D} e^{-t/\theta} \frac{(t/\theta)^{j}}{j!} \left(\frac{\theta}{\Delta(P)}\right)^{D}
 = \mathbb{P}(N \ge D) \left(\frac{\theta}{\Delta(P)}\right)^{D}. \]
Since $t/\theta \ge D$, Lemma \ref{lem:poisson-median} yields $\mathbb{P}(N \ge D) \ge 1/2$, and hence
$\mathscr{P}_t(x,y) \ge \frac{1}{2} \left(\theta/\Delta(P)\right)^{D}$.
Since $x, y \in \mathcal{X}$ are arbitrary, the assertion follows.
\end{proof}

For $\theta = 1/4$, Lemma \ref{lem:Pt-upper} is the estimate in \cite[Eq. (16)]{Salez2024}, and the proof above follows the one there.
We use it with $\theta = 1/(k_{(\mathcal{X}, P)}+3)$; see the proof of Proposition \ref{thm:7.2} below, where the condition $t \ge \mathrm{diam}(\mathcal{X})/(k_{(\mathcal{X}, P)}+3)$ is derived from Proposition \ref{prop:6.3}.

By Lemma \ref{lem:log-Pt-Lip}, Lemma \ref{lem:delta-star} and Lemma \ref{lem:Pt-upper}, we obtain the following assertion, which corresponds to \cite[Lemma 10]{Salez2024}. 

\begin{Prop}[logarithmic gradient estimate]\label{thm:3.14}
If $t \ge \mathrm{diam}(\mathcal{X})/(k_{(\mathcal{X}, P)}+3)$, then for every $o \in \mathcal{X}$,
\[ \left\|\log\frac{\mathscr{P}_t (o,\cdot)}{\pi(\cdot)}\right\|_{\mathrm{Lip}} \le k_{(\mathcal{X}, P)} \left(k_{(\mathcal{X}, P)}+1\right) \log\Delta(P) \]
\[ + k_{(\mathcal{X}, P)} \log\left(e - 1 + (k_{(\mathcal{X}, P)}+3) \log\left(2 (k_{(\mathcal{X}, P)}+3) \Delta(P)\right)\right). \]
In particular,
\[ \left\|\log\frac{\mathscr{P}_t (o,\cdot)}{\pi(\cdot)}\right\|_{\mathrm{Lip}} \le 3k_{(\mathcal{X}, P)}^2(\log\Delta(P) + 4). \]
\end{Prop}

\begin{proof}
Let $k \coloneqq k_{(\mathcal{X}, P)}$, $\Delta \coloneqq \Delta(P)$ and $D \coloneqq \mathrm{diam}(\mathcal{X})$.
Since $k \ge 1$, there exist $x, y \in \mathcal{X}$ with $d_{P}(x \to y) = 1$, and hence $D \ge 1$.
Let $\theta \coloneqq 1/(k+3) \in (0, 1/4]$.
Then $t \ge \theta D$, and by Lemma \ref{lem:Pt-upper}, $\widetilde{\Delta}(\mathscr{P}_t) \le 2 ((k+3)\Delta)^{D}$.
Since $1/t \le (k+3)/D \le k+3$,
\[ \frac{1}{t} \log \widetilde{\Delta}(\mathscr{P}_t) \le \frac{\log 2}{t} + \frac{D}{t} \log((k+3)\Delta) \le (k+3) \log\left(2(k+3)\Delta\right). \]
By this, Lemma \ref{lem:log-Pt-Lip} and Lemma \ref{lem:delta-star}, the first assertion follows.

We show the second assertion.
Let $a \coloneqq e - 1 + (k+3)\log(2(k+3))$ and $b \coloneqq (k+3) \log \Delta$.
Then $b \ge 0$, and $a \ge (k+3) \log 8 \ge k+3$.
By the concavity of $\log$, we have $\log(a + b) \le \log a + b/a \le \log a + \log \Delta$.
Since $\log y \le y - 1$ for $y > 0$, we have $\log(2(k+3)) \le 2(k+3) - 1$, and hence $a \le 2(k+3)^{2} - (k+3) + e - 1 \le 2(k+3)^{2}$.
Therefore, by $\log(k+3) \le k+2$,
\[ \log a \le \log 2 + 2\log(k+3) \le \log 2 + 2(k+2) \le 12k. \]
By these estimates and the first assertion,
\begin{align*} 
\left\|\log\frac{\mathscr{P}_t (o,\cdot)}{\pi(\cdot)}\right\|_{\mathrm{Lip}} 
&\le k \left( (k+1)\log\Delta + \log\Delta + 12k \right) \\
&\le k \left( 3k \log\Delta + 12k \right) = 3k^{2}(\log\Delta + 4). \qedhere 
\end{align*}
\end{proof}

Now we obtain that 
\begin{Prop}\label{thm:5.1}
Assume the Bakry--\'Emery curvature is non-negative. 
Let $\varepsilon \in (0,1/2)$.
Assume that $(k_{(\mathcal{X}, P)}+3)\,t_{\mathrm{mix}}(1-\varepsilon) \ge \mathrm{diam}(\mathcal{X})$.
Then
\[ t_{\mathrm{mix}}(\varepsilon)-t_{\mathrm{mix}}(1-\varepsilon) \le \frac{2t_{\mathrm{rel}}}{\varepsilon^2} \left(1 + 3k_{(\mathcal{X}, P)}^2 \left(\log\Delta(P)+4\right)\sqrt{t_{\mathrm{mix}}(1-\varepsilon)}\right). \]
\end{Prop}

\begin{proof}
By Theorem \ref{thm:Thm5Salez},
\[ t_{\mathrm{mix}}(\varepsilon)-t_{\mathrm{mix}}(1-\varepsilon) \le \frac{2}{\varepsilon^2}\,t_{\mathrm{rel}} \left(1+\sqrt{V_{\mathrm{KL}}^{*}\!\left(t_{\mathrm{mix}}(1-\varepsilon)\right)}\right).\]

By Lemma \ref{lem:semigroup} and Proposition \ref{thm:3.14}, 
it holds that for $t \ge \mathrm{diam}(\mathcal{X})/(k_{(\mathcal{X}, P)}+3)$ and for $f_o=\log\left(\frac{\mathscr{P}_{t}(o,\cdot)}{\pi(\cdot)}\right)$, 
\[ V_{\mathrm{KL}}^{*}(t) \le \max_{o}\left\|\mathscr{P}_{t}(f_o^2)-(\mathscr{P}_{t} f_o)^2\right\|_\infty  \le t\,\max_{o}\|f_o\|_{\mathrm{Lip}}^2 \le t\left(3k_{(\mathcal{X}, P)}^2\left(\log\Delta(P)+4\right)\right)^2. \qedhere \]
\end{proof}

\subsection{Diameter bound}\label{subsec:diameter}

The following corresponds to \cite[Lemma 11]{Salez2024}. 

\begin{Prop}\label{prop:6.3}
For every $\varepsilon \in (0,1)$,
\[ \mathrm{diam}(\mathcal{X}) \le (1+k_{(\mathcal{X}, P)})\left(t_{\mathrm{mix}}(\varepsilon)+\sqrt{\frac{2t_{\mathrm{mix}}(\varepsilon)}{1-\varepsilon}}+k_{(\mathcal{X}, P)}\sqrt{\frac{t_{\mathrm{rel}}}{1-\varepsilon}}\right).\]
\end{Prop}

\begin{proof}
Fix $o \in \mathcal{X}$. 
Let $t \coloneqq t_{\mathrm{mix}}(\varepsilon) \ge 0$. 
Let
\[ A_{o,t,\varepsilon} \coloneqq \left\{x\in\mathcal{X}\ \middle|\ d_P (o \to x) \le t+\sqrt{\frac{2t}{1-\varepsilon}}\right\}. \]

Let $(Y_n)_{n \ge 0}$ be a Markov chain on $\mathcal{X}$ with transition probability $P$ and $Y_0 = o$,
and let $(N_t)_{t \ge 0}$ be a Poisson process with rate $1$ which is independent of $(Y_n)_{n \ge 0}$.
Then the distribution of $Y_{N_t}$ is $\mathscr{P}_t (o, \cdot)$.
Since $P(Y_{n}, Y_{n+1}) > 0$ almost surely for every $n \ge 0$, we see that $d_P (o \to Y_n) \le n$ almost surely for every $n \ge 0$,
and hence $d_P (o \to Y_{N_t}) \le N_t$ almost surely;
that is, the distribution of $d_P (o \to Y_{N_t})$ is stochastically dominated by the Poisson distribution with mean $t$.

Hence
\[ \mathscr{P}_t\left(o, A_{o,t,\varepsilon}^c\right) = \mathbb{P} \left(Y_{N_t} \notin A_{o,t,\varepsilon}\right) \]
\[ = \mathbb{P} \!\left(d_P (o \to Y_{N_t}) > t+\sqrt{\frac{2t}{1-\varepsilon}}\right) \le \mathbb{P}\!\left(N_t >  t+\sqrt{\frac{2t}{1-\varepsilon}}\right). \]

Assume that $t > 0$. 
Since $\mathbb{E}[N_t] = \mathrm{Var}(N_t) = t > 0$, by the Chebyshev inequality,
\[ \mathbb{P}\!\left(N_t > t+\sqrt{\frac{2t}{1-\varepsilon}}\right) \le  \frac{1-\varepsilon}{2}. \]
Therefore, 
\begin{equation}\label{eq:Pt-lower-annulus}
\mathscr{P}_t \left(o,A_{o,t,\varepsilon}\right)\ge \frac{1+\varepsilon}{2}. 
\end{equation}
If $t=0$, then $A_{o,0,\varepsilon} = \{o\}$ and $\mathcal P_0(o,A_{o,0,\varepsilon})=1$, and hence, \eqref{eq:Pt-lower-annulus} holds.

Since $t = t_{\mathrm{mix}}(\varepsilon)$, we obtain that
$\mathscr{P}_t \left(o,A_{o,t,\varepsilon}\right)\le \pi(A_{o,t,\varepsilon})+\varepsilon$. 
By this and \eqref{eq:Pt-lower-annulus}, it holds that 
\[ \pi\left(A_{o,t,\varepsilon}\right) \ge \frac{1-\varepsilon}{2}.\]

Let $f(x)\coloneqq d_P (o \to x)$ for $x \in \mathcal{X}$, and let $U$ be a random variable with values in $\mathcal{X}$ whose distribution is $\pi$.
Then
\begin{equation}\label{eq:dist-half}
\mathbb{P} \left(f(U)\le t+\sqrt{\frac{2t}{1-\varepsilon}}\right) = \pi\left(A_{o,t,\varepsilon}\right) \ge \frac{1-\varepsilon}{2}.
\end{equation}

We remark that for every $x \in \mathcal{X}$, $2\Gamma(f,f)(x) \le \|f\|_{\mathrm{Lip}}^2$.
Hence, $2\mathcal{E}(f) \le \|f\|_{\mathrm{Lip}}^2$.
By \eqref{eq:def-relax} and the estimate that $\|f\|_{\mathrm{Lip}} \le k_{(\mathcal{X}, P)}$,
we obtain that
\[ \mathrm{Var}_\pi (f)\le t_{\mathrm{rel}}\mathcal{E}(f) \le \frac{1}{2}\, t_{\mathrm{rel}} \left\|f \right\|_{\mathrm{Lip}}^2 \le \frac{1}{2}\, t_{\mathrm{rel}} k_{(\mathcal{X}, P)}^2. \]

By this estimate and Chebyshev's inequality, we see that for every $\varepsilon^{\prime} \in (\varepsilon, 1)$, 
\begin{equation*}
\mathbb{P}\left(\left|f(U)- \mathbb{E}[f(U)]\right|\le k_{(\mathcal{X}, P)}\sqrt{\frac{t_{\mathrm{rel}}}{1-\varepsilon^{\prime} }} \right)\ge \frac{1+\varepsilon^{\prime} }{2}.
\end{equation*}

By this and \eqref{eq:dist-half}, 
\[ \left\{f(U)\le t_{\mathrm{mix}}(\varepsilon) + \sqrt{\frac{2 t_{\mathrm{mix}}(\varepsilon)}{1-\varepsilon}}\right\} \cap \left\{\left|f(U)- \mathbb{E}[f(U)]\right|\le k_{(\mathcal{X}, P)}\sqrt{\frac{t_{\mathrm{rel}}}{1-\varepsilon^{\prime} }}\right\} \ne \emptyset. \]

Therefore,
\[  \mathbb{E}\left[d_P (o \to U) \right] = \mathbb{E}[f(U)] \le t_{\mathrm{mix}}(\varepsilon) + \sqrt{\frac{2t_{\mathrm{mix}}(\varepsilon)}{1-\varepsilon}} +k_{(\mathcal{X}, P)}\sqrt{\frac{t_{\mathrm{rel}}}{1-\varepsilon^{\prime} }}. \]
By letting $\varepsilon^{\prime} \to \varepsilon$,
we obtain that
\begin{equation}\label{eq:exp-dist-pi}
\mathbb{E}\left[d_P (o \to U) \right] \le t_{\mathrm{mix}}(\varepsilon) + \sqrt{\frac{2t_{\mathrm{mix}}(\varepsilon) }{1-\varepsilon}} +k_{(\mathcal{X}, P)}\sqrt{\frac{t_{\mathrm{rel}}}{1-\varepsilon}}.
\end{equation}

We see that
$d_P (x \to y) \le k_{(\mathcal{X}, P)}\,d_P (y \to x)$ for $x, y \in \mathcal{X}$,
by induction on the value of $d_P (y\to x)$.
Hence, for every $o^{\prime}, x \in \mathcal{X}$,
\[ d_P (o \to o^{\prime}) \le d_P (o\to x) + d_P (x\to o') \le d_P (o\to x)+k_{(\mathcal{X}, P)}\,d_P (o^{\prime} \to x). \]

By this and \eqref{eq:exp-dist-pi},
\[ d_P (o \to o') \le \mathbb{E}[d_P (o\to U)] + k_{(\mathcal{X}, P)}\, \mathbb{E}\left[d_P (o^{\prime}  \to U)\right] \]
\[ \le (k_{(\mathcal{X}, P)}+1)\left(t_{\mathrm{mix}}(\varepsilon) + \sqrt{\frac{2 t_{\mathrm{mix}}(\varepsilon)}{1-\varepsilon}} +k_{(\mathcal{X}, P)}\sqrt{\frac{t_{\mathrm{rel}}}{1-\varepsilon}} \right). \qedhere \]
\end{proof}

\subsection{Cutoff}\label{subsec:cutoff}

In this subsection, we complete the proof of Theorem \ref{thm:main}.
We first give a lower bound for the relaxation time.

\begin{Lem}\label{lem:trel-lower}
If $|\mathcal{X}| \ge 2$, then $t_{\mathrm{rel}} \ge \frac{1}{2}$.
\end{Lem}

\begin{proof}
It is easy to see that for every $f : \mathcal{X} \to \mathbb{R}$,
\begin{equation}\label{eq:E-var}
\mathcal{E}(f) \le 2\,\mathrm{Var}_{\pi}(f).
\end{equation}
Since $|\mathcal{X}| \ge 2$ and $(\mathcal{X}, P)$ is irreducible,
there exist $x_0, y_0 \in \mathcal{X}$ such that $x_0 \ne y_0$ and $P(x_0, y_0) > 0$.
Let $f_0 \coloneqq \mathbf{1}_{\{x_0\}}$ be the indicator function of $\{x_0\}$.
Then
\[ \mathcal{E}(f_0) \ge \pi(x_0)\,\Gamma(f_0)(x_0) \ge \frac{\pi(x_0)\,P(x_0,y_0)}{2} > 0. \]
By \eqref{eq:def-relax} and \eqref{eq:E-var},
\[ t_{\mathrm{rel}} \ge \frac{\mathrm{Var}_{\pi}(f_0)}{\mathcal{E}(f_0)} \ge \frac{1}{2}. \qedhere \]
\end{proof}

We give an upper bound for the varentropy at the mixing time, which corresponds to \cite[Theorem 6]{Salez2024}. 

\begin{Prop}\label{thm:7.2}
Assume that every $\mathcal{X}_{(n)}$ has non-negative Bakry--\'Emery curvature.
Let $\varepsilon \in (0,1)$.
Assume that
\begin{equation}\label{eq:assump-varentropy}
\lim_{n \to \infty} \frac{k_n}{\sqrt{t_{\mathrm{mix}}^{(n)}(\varepsilon)}} = 0
\quad \text{and} \quad
\lim_{n \to \infty} \frac{k_n^{2}\sqrt{t_{\mathrm{rel}}^{(n)}}}{t_{\mathrm{mix}}^{(n)}(\varepsilon)} = 0.
\end{equation}
Then there exists $N \in \mathbb{N}$ such that for every $n \ge N$,
\[ V_{\mathrm{KL},n}^{*}\!\left(t_{\mathrm{mix}}^{(n)}(\varepsilon)\right) \le 144\,k_n^{4}\left(1+\log \Delta_n\right)^{2} t_{\mathrm{mix}}^{(n)}(\varepsilon). \]
\end{Prop}

\begin{proof}
Let $t_n \coloneqq t_{\mathrm{mix}}^{(n)}(\varepsilon)$.
Since $k_n \ge 1$, the first assumption of \eqref{eq:assump-varentropy} implies that
$\lim_{n \to \infty} t_n = \infty$.
By Proposition \ref{prop:6.3},
\[ \mathrm{diam}(\mathcal{X}_{(n)}) \le (k_n+1)\left(t_n + \sqrt{\frac{2t_n}{1-\varepsilon}} + k_n \sqrt{\frac{t_{\mathrm{rel}}^{(n)}}{1-\varepsilon}}\right). \]
Since $k_n \ge 1$, we see that $k_n + 1 \le 2k_n$, and hence
\[ (k_n+1)\left(\sqrt{\frac{2t_n}{1-\varepsilon}} + k_n \sqrt{\frac{t_{\mathrm{rel}}^{(n)}}{1-\varepsilon}}\right) \le \frac{2\sqrt{2}}{\sqrt{1-\varepsilon}} \left(k_n\sqrt{t_n} + k_n^{2}\sqrt{t_{\mathrm{rel}}^{(n)}}\right). \]
By \eqref{eq:assump-varentropy},
\[ k_n \sqrt{t_n} + k_n^{2}\sqrt{t_{\mathrm{rel}}^{(n)}} = o(t_n), \quad n \to \infty. \]
Therefore there exists $N \in \mathbb{N}$ such that for every $n \ge N$, $t_n \ge \mathrm{diam}(\mathcal{X}_{(n)})/(k_n+3)$.

Let $n \ge N$.
For $o \in \mathcal{X}_{(n)}$, let
$f_o \coloneqq \log\left(\frac{\mathscr{P}_{t_n}^{(n)}(o,\cdot)}{\pi_{(n)}(\cdot)}\right)$.
In the same manner as in the proof of Proposition \ref{thm:5.1},
by Lemma \ref{lem:semigroup} and Proposition \ref{thm:3.14},
\[ V_{\mathrm{KL},n}^{*}(t_n) \le \max_{o \in \mathcal{X}_{(n)}} \left\|\mathscr{P}_{t_n}^{(n)}(f_o^2) - (\mathscr{P}_{t_n}^{(n)} f_o)^2\right\|_{\infty} \]
\[ \le t_n \max_{o \in \mathcal{X}_{(n)}} \|f_o\|_{\mathrm{Lip}}^{2} \le t_n \left(3k_n^{2}\left(\log \Delta_n + 4\right)\right)^{2}. \]
Since $\Delta_n \ge 1$, we see that $\log \Delta_n + 4 \le 4\left(1 + \log \Delta_n\right)$, and hence
\[ V_{\mathrm{KL},n}^{*}(t_n) \le 144\,k_n^{4}\left(1+\log \Delta_n\right)^{2} t_n. \qedhere \]
\end{proof}

We now complete the proof of the main result.

\begin{proof}[Proof of Theorem \ref{thm:main}]
It suffices to show that for every $\varepsilon \in \left(0,1/2\right)$,
\begin{equation}\label{eq:cutoff-goal}
\lim_{n\to\infty}\frac{t_{\mathrm{mix}}^{(n)}(\varepsilon)}{t_{\mathrm{mix}}^{(n)}(1-\varepsilon)}=1.
\end{equation}
Fix $\varepsilon \in \left(0,1/2\right)$.
Since $\varepsilon < 1-\varepsilon$, it holds that
$t_{\mathrm{mix}}^{(n)}(\varepsilon) \ge t_{\mathrm{mix}}^{(n)}(1-\varepsilon)$ for every $n \ge 1$.

Since $k_n \ge 1$, there exist $x, y \in \mathcal{X}_{(n)}$ such that $d_{P_{(n)}}(x \to y) = 1$,
and in particular $|\mathcal{X}_{(n)}| \ge 2$.
Hence, by Lemma \ref{lem:trel-lower} and $\Delta_n \ge 1$,
\begin{equation}\label{eq:lower-half}
k_n^{2}\left(1+\log \Delta_n\right) t_{\mathrm{rel}}^{(n)} \ge t_{\mathrm{rel}}^{(n)} \ge 1/2.
\end{equation}
By this and \eqref{eq:main-assumption},
\[ \frac{1}{2\sqrt{t_{\mathrm{mix}}^{(n)}(1-\varepsilon)}} \le \frac{k_n^{2}\left(1+\log \Delta_n\right) t_{\mathrm{rel}}^{(n)}}{\sqrt{t_{\mathrm{mix}}^{(n)}(1-\varepsilon)}} \to 0, \quad n \to \infty, \]
and hence
\begin{equation}\label{eq:tmix-infty}
\lim_{n\to\infty} t_{\mathrm{mix}}^{(n)}(1-\varepsilon) = \infty.
\end{equation}

We now verify the assumptions of Proposition \ref{thm:7.2} with $1-\varepsilon \in \left(1/2,1\right)$ in place of $\varepsilon$.
Since $k_n \ge 1$ and $2\left(1+\log \Delta_n\right) t_{\mathrm{rel}}^{(n)} \ge 1$ by \eqref{eq:lower-half}, it follows from \eqref{eq:main-assumption} that
\[ \frac{k_n}{\sqrt{t_{\mathrm{mix}}^{(n)}(1-\varepsilon)}} \le \frac{k_n^{2}}{\sqrt{t_{\mathrm{mix}}^{(n)}(1-\varepsilon)}} \le \frac{2k_n^{2}\left(1+\log \Delta_n\right) t_{\mathrm{rel}}^{(n)}}{\sqrt{t_{\mathrm{mix}}^{(n)}(1-\varepsilon)}} \to 0, \quad n \to \infty. \]
Since $t_{\mathrm{rel}}^{(n)} \ge 1/2$, we see that $\sqrt{t_{\mathrm{rel}}^{(n)}} \le \sqrt{2}\,t_{\mathrm{rel}}^{(n)}$.
By \eqref{eq:tmix-infty}, it holds that $t_{\mathrm{mix}}^{(n)}(1-\varepsilon) \ge \sqrt{t_{\mathrm{mix}}^{(n)}(1-\varepsilon)}$ for every sufficiently large $n$, and hence, by \eqref{eq:main-assumption},
\[ \frac{k_n^{2}\sqrt{t_{\mathrm{rel}}^{(n)}}}{t_{\mathrm{mix}}^{(n)}(1-\varepsilon)} \le \frac{\sqrt{2}\,k_n^{2}\left(1+\log \Delta_n\right) t_{\mathrm{rel}}^{(n)}}{\sqrt{t_{\mathrm{mix}}^{(n)}(1-\varepsilon)}} \to 0, \quad n \to \infty. \]
Therefore, by Proposition \ref{thm:7.2}, there exists $N \in \mathbb{N}$ such that for every $n \ge N$,
\[ V_{\mathrm{KL},n}^{*}\!\left(t_{\mathrm{mix}}^{(n)}(1-\varepsilon)\right) \le 144\,k_n^{4}\left(1+\log \Delta_n\right)^{2} t_{\mathrm{mix}}^{(n)}(1-\varepsilon), \]
that is,
\[ \sqrt{V_{\mathrm{KL},n}^{*}\!\left(t_{\mathrm{mix}}^{(n)}(1-\varepsilon)\right)} \le 12\,k_n^{2}\left(1+\log \Delta_n\right)\sqrt{t_{\mathrm{mix}}^{(n)}(1-\varepsilon)}. \]

By Theorem \ref{thm:Thm5Salez}, for every $n \ge N$,
\[ 0 \le t_{\mathrm{mix}}^{(n)}(\varepsilon) - t_{\mathrm{mix}}^{(n)}(1-\varepsilon) \le \frac{2t_{\mathrm{rel}}^{(n)}}{\varepsilon^{2}} \left(1 + 12\,k_n^{2}\left(1+\log \Delta_n\right)\sqrt{t_{\mathrm{mix}}^{(n)}(1-\varepsilon)}\right). \]
By \eqref{eq:tmix-infty}, $t_{\mathrm{mix}}^{(n)}(\varepsilon) \ge t_{\mathrm{mix}}^{(n)}(1-\varepsilon) > 0$ for every sufficiently large $n$.
Hence, dividing by $t_{\mathrm{mix}}^{(n)}(\varepsilon)$, we obtain that for every sufficiently large $n$,
\begin{equation}\label{eq:ratio-bound}
0 \le 1 - \frac{t_{\mathrm{mix}}^{(n)}(1-\varepsilon)}{t_{\mathrm{mix}}^{(n)}(\varepsilon)}
\le \frac{2t_{\mathrm{rel}}^{(n)}}{\varepsilon^{2}\,t_{\mathrm{mix}}^{(n)}(\varepsilon)}
+ \frac{24\,k_n^{2}\left(1+\log \Delta_n\right) t_{\mathrm{rel}}^{(n)} \sqrt{t_{\mathrm{mix}}^{(n)}(1-\varepsilon)}}{\varepsilon^{2}\,t_{\mathrm{mix}}^{(n)}(\varepsilon)}.
\end{equation}
Since $k_n^{2}\left(1+\log \Delta_n\right) \ge 1$ and
$t_{\mathrm{mix}}^{(n)}(\varepsilon) \ge t_{\mathrm{mix}}^{(n)}(1-\varepsilon) \ge \sqrt{t_{\mathrm{mix}}^{(n)}(1-\varepsilon)}$ for every sufficiently large $n$, it follows from \eqref{eq:main-assumption} that
\[ \frac{t_{\mathrm{rel}}^{(n)}}{t_{\mathrm{mix}}^{(n)}(\varepsilon)} \le \frac{k_n^{2}\left(1+\log \Delta_n\right) t_{\mathrm{rel}}^{(n)}}{\sqrt{t_{\mathrm{mix}}^{(n)}(1-\varepsilon)}} \to 0, \quad n \to \infty. \]
Furthermore, by $t_{\mathrm{mix}}^{(n)}(\varepsilon) \ge t_{\mathrm{mix}}^{(n)}(1-\varepsilon)$ and \eqref{eq:main-assumption},
\[ \frac{k_n^{2}\left(1+\log \Delta_n\right) t_{\mathrm{rel}}^{(n)} \sqrt{t_{\mathrm{mix}}^{(n)}(1-\varepsilon)}}{t_{\mathrm{mix}}^{(n)}(\varepsilon)}
\le \frac{k_n^{2}\left(1+\log \Delta_n\right) t_{\mathrm{rel}}^{(n)}}{\sqrt{t_{\mathrm{mix}}^{(n)}(1-\varepsilon)}} \to 0, \quad n \to \infty. \]
By these estimates and \eqref{eq:ratio-bound}, \eqref{eq:cutoff-goal} holds.
Since $\varepsilon \in \left(0,1/2\right)$ is arbitrary, a cutoff occurs for $\left(\mathcal{X}_{(n)} \right)_n$.
\end{proof}

We give two remarks on the quantity $\Delta_n$ appearing in \eqref{eq:main-assumption}.

\begin{Rem}\label{rem:delta-symmetric}
Assume that $|\mathcal{X}| \ge 3$ and that the support of $P$ is symmetric, as in \cite{Salez2024}.
Then $\Delta(P) \ge 2$, and in particular $\log \Delta(P) \ge \log 2 > 0$.

Assume that $\Delta(P) < 2$.
Then, for every $x, y \in \mathcal{X}$ with $P(x,y) > 0$ and $x \ne y$,
we see that $P(x,y) \ge \frac{1}{\Delta(P)} > \frac{1}{2}$.
Since $\sum_{y \in \mathcal{X}} P(x,y) = 1$,
for every $x \in \mathcal{X}$ there exists at most one $y \ne x$ such that $P(x,y) > 0$,
and, by irreducibility, there exists at least one such $y$.
Hence for every $x \in \mathcal{X}$ there exists a unique $\sigma(x) \ne x$ such that $P(x, \sigma(x)) > 0$.
Since the support of $P$ is symmetric, we see that $P(\sigma(x), x) > 0$, and hence $\sigma(\sigma(x)) = x$.
Then, for every $x \in \mathcal{X}$, every path starting from $x$ stays in $\{x, \sigma(x)\}$,
which contradicts irreducibility and $|\mathcal{X}| \ge 3$.
\end{Rem}

\begin{Rem}\label{rem:delta-one}
In Theorem \ref{thm:main}, we do not assume that $\Delta_n \ge 2$,
while it is obvious that $\Delta_n \ge 1$.
The bound $\Delta(P) \ge 1$ is attained:
if $\mathcal{X} = \{x_1, x_2\}$ and $P(x_1, x_2) = P(x_2, x_1) = 1$, then $\Delta(P) = 1$.
The support of $P$ may also be non-symmetric in this situation:
for the directed cycle of length $m \ge 3$, that is,
$\mathcal{X} = \{x_1, \dots, x_m\}$ and $P(x_i, x_{i+1}) = 1$ for $1 \le i \le m$, where $x_{m+1} \coloneqq x_1$,
we see that $\Delta(P) = 1$.
This is the reason why we adopt the factor $1 + \log \Delta_n$ instead of $\log \Delta_n$ in \eqref{eq:main-assumption}.
\end{Rem}

The assumption \eqref{eq:main-assumption} may be complicated to verify directly.
We conclude this section by giving a sufficient condition for \eqref{eq:main-assumption} in terms of the sizes $|\mathcal{X}_{(n)}|$.
We first give a version of the Moore bound for directed graphs.

\begin{Lem}[Moore bound]\label{lem:moore}
\[ \mathrm{diam}(\mathcal{X}) \ge \frac{\log |\mathcal{X}|}{1 + \log \Delta(P)}. \]
\end{Lem}

\begin{proof}
If $|\mathcal{X}| = 1$, then both sides are equal to $0$.
Assume that $|\mathcal{X}| \ge 2$.
Let $N \coloneqq \mathrm{diam}(\mathcal{X}) \ge 1$ and $D \coloneqq \lfloor \Delta(P) \rfloor \ge 1$.
If $P(x,y) > 0$ and $x \ne y$, then $P(x,y) \ge 1/\Delta(P)$.
Since $\sum_{y \in \mathcal{X}} P(x,y) = 1$, it holds that for every $x \in \mathcal{X}$,
\[ \left|\left\{y \in \mathcal{X} \,\middle|\, P(x,y) > 0, \ y \ne x \right\}\right| \le \Delta(P), \]
and hence the left-hand side is at most $D$.

Fix $x_0 \in \mathcal{X}$ and let
\[ S_k \coloneqq \left\{y \in \mathcal{X} \,\middle|\, d_{P}(x_0 \to y) = k \right\}, \quad k \in \mathbb{N}_0. \]
Let $k \ge 1$ and $y \in S_k$.
Then there exists $z \in \mathcal{X}$ such that $d_{P}(x_0 \to z) = k-1$ and $P(z,y) > 0$,
and we see that $z \ne y$, since otherwise $d_{P}(x_0 \to y) \le k-1$.
Hence
\[ S_k \subset \bigcup_{z \in S_{k-1}} \left\{w \in \mathcal{X} \,\middle|\, P(z,w) > 0, \ w \ne z \right\}, \]
and therefore $|S_k| \le D\,|S_{k-1}|$ for every $k \ge 1$.
By induction, $|S_k| \le D^{k}$ for every $k \in \mathbb{N}_0$.
Since $d_{P}(x_0 \to y) \le N$ for every $y \in \mathcal{X}$,
\[ |\mathcal{X}| = \sum_{k=0}^{N} |S_k| \le \sum_{k=0}^{N} D^{k} \le (1+D)^{N}, \]
where the last inequality follows by induction on $N \ge 1$: indeed,
\[ \sum_{k=0}^{N} D^{k} = 1 + D \sum_{k=0}^{N-1} D^{k} \le (1+D)^{N-1} + D\,(1+D)^{N-1} = (1+D)^{N}. \]
Since $\Delta(P) \ge 1$, we see that $1 + D \le 2\Delta(P)$, and hence
\[ \log |\mathcal{X}| \le N \log(1+D) \le N \left(\log 2 + \log \Delta(P)\right) \le N \left(1 + \log \Delta(P)\right). \qedhere \]
\end{proof}

\begin{Prop}\label{prop:sufficient}
Assume that
\begin{equation}\label{eq:suff}
\lim_{n \to \infty} t_{\mathrm{rel}}^{(n)} \sqrt{\frac{(1+k_n)^{5}\left(1+\log \Delta_n\right)^{3}}{\log |\mathcal{X}_{(n)}|}} = 0.
\end{equation}
Then, for every $\varepsilon \in \left[1/2, 1\right)$,
\begin{equation}\label{eq:suff-2} 
\lim_{n \to \infty} \frac{k_n^{2}\left(1+\log \Delta_n\right) t_{\mathrm{rel}}^{(n)}}{\sqrt{t_{\mathrm{mix}}^{(n)}(\varepsilon)}} = 0. 
\end{equation}
In particular, \eqref{eq:main-assumption} holds for every $\varepsilon \in \left(0,1/2\right)$, and hence,
if every $\mathcal{X}_{(n)}$ has non-negative Bakry--\'Emery curvature,
then a cutoff occurs for $\left(\mathcal{X}_{(n)}\right)_n$.
\end{Prop}

\begin{proof}
As in the proof of Theorem \ref{thm:main}, we see that $|\mathcal{X}_{(n)}| \ge 2$,
and hence $\log |\mathcal{X}_{(n)}| \ge \log 2 > 0$ and, by Lemma \ref{lem:trel-lower}, $t_{\mathrm{rel}}^{(n)} \ge 1/2$.
Let
\[ R_n \coloneqq \frac{\log |\mathcal{X}_{(n)}|}{(1+k_n)^{5}\left(1+\log \Delta_n\right)^{3}}
\quad \text{and} \quad
L_n \coloneqq \frac{\log |\mathcal{X}_{(n)}|}{(1+k_n)\left(1+\log \Delta_n\right)}. \]
Then \eqref{eq:suff} states that $\lim_{n \to \infty} t_{\mathrm{rel}}^{(n)}/\sqrt{R_n} = 0$, and
\begin{equation}\label{eq:LnRn}
L_n = (1+k_n)^{4}\left(1+\log \Delta_n\right)^{2} R_n \ge R_n.
\end{equation}
Since $t_{\mathrm{rel}}^{(n)} \ge 1/2$, it follows from \eqref{eq:suff} that
\begin{equation}\label{eq:Rn-infty}
\lim_{n \to \infty} R_n = \infty,
\end{equation}
and hence, by \eqref{eq:LnRn}, $\lim_{n \to \infty} L_n = \infty$.
Furthermore, by \eqref{eq:LnRn} and $k_n \le 1+k_n$,
\begin{equation}\label{eq:kn-sqrt-trel}
\frac{k_n \sqrt{t_{\mathrm{rel}}^{(n)}}}{L_n} \le \frac{\sqrt{t_{\mathrm{rel}}^{(n)}}}{(1+k_n)^{3}\left(1+\log \Delta_n\right)^{2} R_n} \le \frac{\sqrt{t_{\mathrm{rel}}^{(n)}}}{R_n} \to 0, \quad n \to \infty.
\end{equation}
Here the convergence in \eqref{eq:kn-sqrt-trel} follows from \eqref{eq:suff} and \eqref{eq:Rn-infty}, since
\[ \frac{\sqrt{t_{\mathrm{rel}}^{(n)}}}{R_n} = \left(\frac{t_{\mathrm{rel}}^{(n)}}{\sqrt{R_n}}\right)^{1/2} \frac{1}{R_n^{3/4}}. \]

Fix $\varepsilon \in \left[1/2, 1\right)$ and let $t_n \coloneqq t_{\mathrm{mix}}^{(n)}(\varepsilon)$ and $c_{\varepsilon} \coloneqq \sqrt{\frac{2}{1-\varepsilon}}$.
By applying Proposition \ref{prop:6.3} and Lemma \ref{lem:moore},
\[ L_n \le t_n + \sqrt{\frac{2t_n}{1-\varepsilon}} + k_n \sqrt{\frac{t_{\mathrm{rel}}^{(n)}}{1-\varepsilon}} \le t_n + c_{\varepsilon}\sqrt{t_n} + c_{\varepsilon}\,k_n \sqrt{t_{\mathrm{rel}}^{(n)}}. \]
By \eqref{eq:kn-sqrt-trel}, it holds that $c_{\varepsilon}\,k_n \sqrt{t_{\mathrm{rel}}^{(n)}} \le \frac{L_n}{2}$ for every sufficiently large $n$, and hence
\[ t_n + c_{\varepsilon}\sqrt{t_n} \ge \frac{L_n}{2} \]
for every sufficiently large $n$.
If $t_n \le 1$, then $t_n + c_{\varepsilon}\sqrt{t_n} \le 1 + c_{\varepsilon}$.
Since $\lim_{n \to \infty} L_n = \infty$, we see that $t_n > 1$ for every sufficiently large $n$,
and hence $\sqrt{t_n} \le t_n$ and
\[ t_n \ge \frac{L_n}{2\left(1 + c_{\varepsilon}\right)} \]
for every sufficiently large $n$.
Therefore, by \eqref{eq:LnRn}, it holds that for every sufficiently large $n$,
\begin{align*} \frac{\left(k_n^{2}\left(1+\log \Delta_n\right) t_{\mathrm{rel}}^{(n)}\right)^{2}}{t_n}
&\le \frac{2\left(1 + c_{\varepsilon}\right) k_n^{4}\left(1+\log \Delta_n\right)^{2} \left(t_{\mathrm{rel}}^{(n)}\right)^{2}}{L_n} \\
&\le 2\left(1 + c_{\varepsilon}\right) \left(\frac{t_{\mathrm{rel}}^{(n)}}{\sqrt{R_n}}\right)^{2} \to 0, \quad n \to \infty, 
\end{align*}
where the convergence follows from \eqref{eq:suff}.
Hence \eqref{eq:suff-2} holds. 
For every $\varepsilon \in \left(0,1/2\right)$, applying this with $1-\varepsilon \in \left(1/2,1\right)$ in place of $\varepsilon$ yields \eqref{eq:main-assumption}.
The last assertion follows from Theorem \ref{thm:main}.
\end{proof}

\subsection{Discrete time}\label{subsec:discrete}

In this subsection, we prove Theorem \ref{thm:discrete}.
Throughout this subsection, we fix $\delta \in \left(0,1\right)$ and assume that $(\mathcal{X}, P)$ is $\delta$-lazy.

Let $Q \coloneqq (1-\delta)^{-1} (P - \delta I)$.
Then, $P = \delta I + (1-\delta) Q$,
that is, $P$ is the $\delta$-lazy walk of $Q$ in the sense of \cite{CSC2013}.
We denote by $(\mathscr{P}_t^{Q})_{t \ge 0}$ the continuous-time semigroup associated with $(\mathcal{X}, Q)$, that is,
\[ \mathscr{P}_t^{Q} (x,y) \coloneqq e^{-t} \sum_{n=0}^{\infty} \frac{t^n}{n!} \, Q^n (x,y), \quad x,y \in \mathcal{X},\ t \ge 0, \]
and define $t_{\mathrm{mix}}^{Q}(\cdot)$ by replacing $\mathscr{P}_t$ with $\mathscr{P}_t^{Q}$ in the definition of $t_{\mathrm{mix}}(\cdot)$.

\begin{Lem}\label{lem:lazy-comparison}
The following hold.\\
(1) $Q$ is a stochastic matrix on $\mathcal{X}$, $(\mathcal{X}, Q)$ is irreducible, and the stationary distribution of $Q$ is $\pi$.\\
(2) $d_{Q}(x \to y) = d_{P}(x \to y)$ for every $x, y \in \mathcal{X}$. In particular, $k_{(\mathcal{X}, Q)} = k_{(\mathcal{X}, P)}$.\\
(3) $\Delta(Q) = (1-\delta)\,\Delta(P) \le \Delta(P)$.\\
(4) The relaxation time of $(\mathcal{X}, Q)$ is equal to $(1-\delta)\,t_{\mathrm{rel}}$.\\
(5) $\textup{BE}(\mathcal{X}, Q) = (1-\delta)^{-1}\,\textup{BE}(\mathcal{X}, P)$.\\
(6) $\mathscr{P}_t^{Q} = \mathscr{P}_{t/(1-\delta)}$ for every $t \ge 0$. In particular, $t_{\mathrm{mix}}^{Q}(\varepsilon) = (1-\delta)\,t_{\mathrm{mix}}(\varepsilon)$ for every $\varepsilon \in (0,1)$.
\end{Lem}

\begin{proof}
(1)
Since $(\mathcal{X}, P)$ is $\delta$-lazy, we see that
$Q(x,x) = (1-\delta)^{-1}\left(P(x,x) - \delta\right) \ge 0$ for every $x \in \mathcal{X}$,
and $Q(x,y) = (1-\delta)^{-1} P(x,y) \ge 0$ for every $x, y \in \mathcal{X}$ with $x \ne y$.
Furthermore, $\sum_{y \in \mathcal{X}} Q(x,y) = (1-\delta)^{-1}(1 - \delta) = 1$ for every $x \in \mathcal{X}$, and
\[ \pi Q = (1-\delta)^{-1}\left(\pi P - \delta \pi\right) = (1-\delta)^{-1}\left(\pi - \delta \pi\right) = \pi. \]
The irreducibility of $(\mathcal{X}, Q)$ follows from (2).

(2)
Since $Q$ commutes with $I$, it holds that for every $k \in \mathbb{N}_0$,
\begin{equation}\label{eq:Pk-Qj}
P^{k} = \left(\delta I + (1-\delta) Q\right)^{k} = \sum_{j=0}^{k} \binom{k}{j} \delta^{k-j} (1-\delta)^{j}\, Q^{j}.
\end{equation}
Since $\delta \in \left(0,1\right)$, all the coefficients in \eqref{eq:Pk-Qj} are positive, and hence,
for every $x, y \in \mathcal{X}$ and $k \in \mathbb{N}_0$,
$P^{k}(x,y) > 0$ if and only if $Q^{j}(x,y) > 0$ for some $j \in \{0, 1, \dots, k\}$.
If $k = d_{P}(x \to y)$, then $P^{k}(x,y) > 0$, and hence $Q^{j}(x,y) > 0$ for some $j \le k$, which yields $d_{Q}(x \to y) \le d_{P}(x \to y)$.
Conversely, if $j = d_{Q}(x \to y)$, then $P^{j}(x,y) \ge (1-\delta)^{j} Q^{j}(x,y) > 0$, and hence $d_{P}(x \to y) \le d_{Q}(x \to y)$.
Since $k_{(\mathcal{X}, P)}$ and $k_{(\mathcal{X}, Q)}$ are determined by $d_{P}$ and $d_{Q}$ respectively, we obtain $k_{(\mathcal{X}, Q)} = k_{(\mathcal{X}, P)}$.

(3)
For $x \ne y$, we see that $Q(x,y) > 0$ if and only if $P(x,y) > 0$, and in this case $\frac{1}{Q(x,y)} = \frac{1-\delta}{P(x,y)}$.
Taking the maximum yields $\Delta(Q) = (1-\delta)\,\Delta(P)$.

(4)
Let $\Gamma^{Q}$ and $\mathcal{E}^{Q}$ denote the carr\'e du champ and the Dirichlet form defined for $(\mathcal{X}, Q)$, respectively.
Since the terms with $y = x$ vanish in the definition of $\Gamma$,
\begin{equation}\label{eq:Gamma-Q}
\Gamma^{Q}(f,g)(x) = \frac{1}{2} \sum_{y \ne x} \left(f(x)-f(y)\right)\left(g(x)-g(y)\right) Q(x,y) = (1-\delta)^{-1}\,\Gamma(f,g)(x),
\end{equation}
and hence $\mathcal{E}^{Q}(f) = (1-\delta)^{-1}\,\mathcal{E}(f)$ for every $f \colon \mathcal{X} \to \mathbb{R}$.
By \eqref{eq:def-relax}, the relaxation time of $(\mathcal{X}, Q)$ is equal to $(1-\delta)\,t_{\mathrm{rel}}$.

(5)
Let $\mathcal{L}^{Q}$ denote the Laplacian defined for $(\mathcal{X}, Q)$.
In the same manner as in \eqref{eq:Gamma-Q}, we see that $\mathcal{L}^{Q} f = (1-\delta)^{-1}\,\mathcal{L} f$.
By this and \eqref{eq:Gamma-Q}, for every $\kappa \in \mathbb{R}$, $f \colon \mathcal{X} \to \mathbb{R}$ and $x \in \mathcal{X}$,
\[ \frac{1}{2} \mathcal{L}^{Q}\Gamma^{Q}(f)(x) - \Gamma^{Q}(f, \mathcal{L}^{Q} f)(x) - \frac{\kappa}{1-\delta}\,\Gamma^{Q}(f)(x) \]
\[ = (1-\delta)^{-2}\left(\frac{1}{2} \mathcal{L}\Gamma(f)(x) - \Gamma(f, \mathcal{L} f)(x) - \kappa\,\Gamma(f)(x)\right). \]
Hence $\kappa$ satisfies the defining inequality of $\textup{BE}(\mathcal{X}, P)$ if and only if $\kappa/(1-\delta)$ satisfies that of $\textup{BE}(\mathcal{X}, Q)$.
Taking the suprema yields the assertion.

(6)
Since $tQ$ and $-tI$ commute, we see that
$\mathscr{P}_t^{Q} = e^{-t}\sum_{n=0}^{\infty} \frac{t^n}{n!} Q^n = e^{t(Q - I)}$, and similarly $\mathscr{P}_t = e^{t(P-I)}$.
Since $Q - I = (1-\delta)^{-1}(P - I)$,
\[ \mathscr{P}_t^{Q} = \exp\left(\frac{t}{1-\delta}\left(P - I\right)\right) = \mathscr{P}_{t/(1-\delta)}. \]
Hence, for every $t \ge 0$,
$\max_{x \in \mathcal{X}}\left\|\mathscr{P}_t^{Q}(x, \cdot) - \pi\right\|_{\mathrm{TV}} = \max_{x \in \mathcal{X}}\left\|\mathscr{P}_{t/(1-\delta)}(x, \cdot) - \pi\right\|_{\mathrm{TV}}$,
and therefore $t_{\mathrm{mix}}^{Q}(\varepsilon) = (1-\delta)\,t_{\mathrm{mix}}(\varepsilon)$ for every $\varepsilon \in (0,1)$.
\end{proof}

Now we proceed to the proof. 

\begin{proof}[Proof of Theorem \ref{thm:discrete}]
For $n \ge 1$, let $Q_{(n)} \coloneqq (1-\delta)^{-1}\left(P_{(n)} - \delta I\right)$.
We denote by $k_n^{Q}$, $\Delta_n^{Q}$, $t_{\mathrm{rel}}^{Q,(n)}$, $(\mathscr{P}_t^{Q,(n)})_{t \ge 0}$ and $t_{\mathrm{mix}}^{Q,(n)}(\cdot)$
the quantities $k_{(\mathcal{X}, Q)}$, $\Delta(Q)$, the relaxation time, the continuous-time semigroup and its mixing time defined for $(\mathcal{X}_{(n)}, Q_{(n)})$, respectively.

By Lemma \ref{lem:lazy-comparison} (1) and (5),
$\left(\mathcal{X}_{(n)},\,Q_{(n)},\pi_{(n)}\right), n \ge 1$, is a sequence of finite irreducible Markov chains with non-negative Bakry--\'Emery curvature.
By Lemma \ref{lem:lazy-comparison} (2), (3), (4) and (6), for every $\varepsilon \in \left(0,1/2\right)$,
\begin{align*} 
\frac{(k_n^{Q})^{2}\left(1+\log \Delta_n^{Q}\right) t_{\mathrm{rel}}^{Q,(n)}}{\sqrt{t_{\mathrm{mix}}^{Q,(n)}(1-\varepsilon)}}
&= \frac{k_n^{2}\left(1+\log\left((1-\delta)\Delta_n\right)\right) (1-\delta)\,t_{\mathrm{rel}}^{(n)}}{\sqrt{(1-\delta)\,t_{\mathrm{mix}}^{(n)}(1-\varepsilon)}} \\
&\le \sqrt{1-\delta}\,\frac{k_n^{2}\left(1+\log \Delta_n\right) t_{\mathrm{rel}}^{(n)}}{\sqrt{t_{\mathrm{mix}}^{(n)}(1-\varepsilon)}}, 
\end{align*}
which converges to $0$ as $n \to \infty$ by \eqref{eq:main-assumption}.
Hence the sequence $\left(\mathcal{X}_{(n)},\,Q_{(n)},\pi_{(n)}\right)_{n}$ satisfies the assumption of Theorem \ref{thm:main},
and therefore, by Theorem \ref{thm:main}, for every $\varepsilon \in (0,1)$,
\begin{equation}\label{eq:Q-cutoff}
\lim_{n\to\infty} \frac{t_{\mathrm{mix}}^{Q,(n)}(\varepsilon)}{t_{\mathrm{mix}}^{Q,(n)}(1-\varepsilon)} = 1.
\end{equation}
Furthermore, in the same manner as in the proof of Theorem \ref{thm:main} (see \eqref{eq:tmix-infty}),
we see that $\lim_{n\to\infty} t_{\mathrm{mix}}^{Q,(n)}(\eta) = \infty$ for every $\eta \in \left(1/2,1\right)$.
Since $t_{\mathrm{mix}}^{Q,(n)}(\cdot)$ is non-increasing, it holds that
\begin{equation}\label{eq:Q-tmix-infty}
\lim_{n\to\infty} t_{\mathrm{mix}}^{Q,(n)}(\varepsilon) = \infty \quad \text{for every } \varepsilon \in (0,1).
\end{equation}

Let $F \coloneqq \left\{\left(\mathcal{X}_{(n)},\,Q_{(n)},\pi_{(n)}\right) \right\}_{n \ge 1}$.
In the notation of \cite{CSC2013},
the family of the $\delta$-lazy walks associated with $F$ is
\[ F_{\delta} = \left\{\left(\mathcal{X}_{(n)},\,\delta I + (1-\delta) Q_{(n)},\pi_{(n)}\right) \right\}_{n \ge 1} = \left\{\left(\mathcal{X}_{(n)},\,P_{(n)},\pi_{(n)}\right)  \right\}_{n \ge 1}, \]
and the family of the continuous-time chains associated with $F$ is
$F_{c} = \left\{\left(\mathcal{X}_{(n)},\,(\mathscr{P}_t^{Q,(n)})_{t \ge 0},\pi_{(n)}\right)  \right\}_{n \ge 1}$.
The total variation mixing times of the $n$-th chains of $F_{\delta}$ and $F_{c}$ are $t_{\mathrm{mix},d}^{(n)}(\cdot)$ and $t_{\mathrm{mix}}^{Q,(n)}(\cdot)$, respectively.

We now show that $F_{c}$ has a total variation cutoff in the sense of \cite[Definition 2.1]{CSC2013}.
Let $0 < \varepsilon < \eta < 1$ and $\alpha \coloneqq \min\{\varepsilon, 1-\eta\}$.
Then $\alpha \in \left(0, 1/2\right]$ and $\alpha \le \varepsilon < \eta \le 1-\alpha$, and hence, for every sufficiently large $n$,
\[ 1 \le \frac{t_{\mathrm{mix}}^{Q,(n)}(\varepsilon)}{t_{\mathrm{mix}}^{Q,(n)}(\eta)} \le \frac{t_{\mathrm{mix}}^{Q,(n)}(\alpha)}{t_{\mathrm{mix}}^{Q,(n)}(1-\alpha)}. \]
By this and \eqref{eq:Q-cutoff},
\begin{equation}\label{eq:Q-ratio-pairs}
\lim_{n\to\infty} \frac{t_{\mathrm{mix}}^{Q,(n)}(\varepsilon)}{t_{\mathrm{mix}}^{Q,(n)}(\eta)} = 1 \quad \text{for every } 0 < \varepsilon < \eta < 1.
\end{equation}
By \eqref{eq:Q-tmix-infty}, \eqref{eq:Q-ratio-pairs} and \cite[Remark 2.1]{CSC2013},
$F_{c}$ has a total variation cutoff.

By \eqref{eq:Q-tmix-infty}, the assumption of \cite[Theorem 3.1]{CSC2013} is satisfied for $F$.
Therefore, by \cite[Theorem 3.1]{CSC2013}, $F_{\delta}$ has a total variation cutoff, and furthermore,
\begin{equation}\label{eq:CSC-ratio}
\lim_{n\to\infty} \frac{t_{\mathrm{mix}}^{Q,(n)}(\varepsilon)}{t_{\mathrm{mix},d}^{(n)}(\varepsilon)} = 1-\delta \quad \text{for every } \varepsilon \in (0,1).
\end{equation}
By \eqref{eq:Q-tmix-infty} and \eqref{eq:CSC-ratio},
we see that $\lim_{n\to\infty} t_{\mathrm{mix},d}^{(n)}(\varepsilon) = \infty$ for every $\varepsilon \in (0,1)$;
in particular, $t_{\mathrm{mix},d}^{(n)}(\varepsilon) > 0$ for every sufficiently large $n$.
Hence, for every $\varepsilon \in (0,1)$,
\[ \frac{t_{\mathrm{mix},d}^{(n)}(\varepsilon)}{t_{\mathrm{mix},d}^{(n)}(1-\varepsilon)}
= \frac{t_{\mathrm{mix},d}^{(n)}(\varepsilon)}{t_{\mathrm{mix}}^{Q,(n)}(\varepsilon)} \cdot
\frac{t_{\mathrm{mix}}^{Q,(n)}(\varepsilon)}{t_{\mathrm{mix}}^{Q,(n)}(1-\varepsilon)} \cdot
\frac{t_{\mathrm{mix}}^{Q,(n)}(1-\varepsilon)}{t_{\mathrm{mix},d}^{(n)}(1-\varepsilon)} \]
\[ \to \frac{1}{1-\delta} \cdot 1 \cdot (1-\delta) = 1, \quad n \to \infty, \]
by \eqref{eq:Q-cutoff} and \eqref{eq:CSC-ratio}.
This completes the proof.
\end{proof}

\begin{Rem}\label{rem:discrete-continuous}
By Lemma \ref{lem:lazy-comparison} (6) and \eqref{eq:CSC-ratio},
under the assumptions of Theorem \ref{thm:discrete},
\[ \lim_{n\to\infty} \frac{t_{\mathrm{mix}}^{(n)}(\varepsilon)}{t_{\mathrm{mix},d}^{(n)}(\varepsilon)} = 1 \quad \text{for every } \varepsilon \in (0,1), \]
that is, the discrete mixing time of $P_{(n)}$ is asymptotically equivalent to the mixing time of the continuous-time chain associated with $(\mathcal{X}_{(n)}, P_{(n)})$.
\end{Rem}

\section{Examples}\label{sec:example}

\subsection{Abelian groups}\label{subsec:abelian}

We show that the Bakry--\'Emery curvature of the random walk on a finite Abelian group associated with a generating set is non-negative,
where the generating set is not assumed to be symmetric with respect to the group operation.
This extends the finite-group case of \cite[Theorem 2.3]{KKRT2016}, in which the generating set is assumed to be symmetric.

In this subsection, we write finite Abelian groups {\it multiplicatively}.
Let $G$ be a finite Abelian group and let $S$ be a non-empty generating set of $G$.
Define the stochastic matrix $P_{S}$ associated with $S$ on $G$ by
\begin{equation}\label{eq:tr-pb-def} 
P_{S}(x, xs) \coloneqq \frac{1}{|S|}, \quad x \in G,\ s \in S, 
\end{equation}
and $P_{S}(x,y) \coloneqq 0$ if $y \notin xS \coloneqq \left\{xs \,\middle|\, s \in S\right\}$. 
Let $S^{-1} \coloneqq \left\{s^{-1} \,\middle|\, s \in S\right\}$. 

For a stochastic matrix $R$ on a finite set $\mathcal{X}$,
we denote by $\mathcal{L}^{R}$ and $\Gamma^{R}$ the Laplacian and the carr\'e du champ defined for $(\mathcal{X}, R)$,
as in the proof of Lemma \ref{lem:lazy-comparison},
and define the iterated gradient by
\[ \Gamma_{2}^{R}(f)(x) \coloneqq \frac{1}{2} \mathcal{L}^{R}\Gamma^{R}(f)(x) - \Gamma^{R}(f, \mathcal{L}^{R} f)(x), \quad f \colon \mathcal{X} \to \mathbb{R},\ x \in \mathcal{X}. \]
By the definition of the Bakry--\'Emery curvature,
if $\Gamma_{2}^{R}(f)(x) \ge 0$ for every $f \colon \mathcal{X} \to \mathbb{R}$ and $x \in \mathcal{X}$,
then $\textup{BE}(\mathcal{X}, R) \ge 0$. 

The following corresponds to \cite[Theorem 2.3]{KKRT2016}. 

\begin{Thm}\label{thm:abelian}
Let $G$ be a finite Abelian group and let $S$ be a non-empty generating set of $G$. 
Then $(G, P_{S})$ is a finite irreducible Markov chain whose stationary distribution is the uniform distribution on $G$, and
$\textup{BE}(G, P_{S}) \ge 0$.
\end{Thm}

We do not assume that $S$ is symmetric, that is, it can happen that $S \ne S^{-1}$. 
In order to prove Theorem \ref{thm:abelian}, we use the following lemma, which holds for every stochastic matrix.

\begin{Lem}\label{lem:gamma2-formula}
Let $R$ be a stochastic matrix on a finite set $\mathcal{X}$.
Let $x \in \mathcal{X}$ and $f \colon \mathcal{X} \to \mathbb{R}$ with $f(x) = 0$.
Then
\[ 2\,\Gamma_{2}^{R}(f)(x) = \left(\sum_{y \in \mathcal{X}} R(x,y) f(y)\right)^{2} - \sum_{y \in \mathcal{X}} R(x,y) f(y)^{2} \]
\[ + \frac{1}{2} \sum_{y, z \in \mathcal{X}} R(x,y) R(y,z) \left(f(z) - 2f(y)\right)^{2}. \]
\end{Lem}

\begin{proof}
Since $f(x) = 0$, we see that
\[ \Gamma^{R}(f)(x) = \frac{1}{2} \sum_{y \in \mathcal{X}} R(x,y) f(y)^{2}, \quad \mathcal{L}^{R} f(x) = \sum_{y \in \mathcal{X}} R(x,y) f(y), \]
and hence
\begin{align*} 
\mathcal{L}^{R}\Gamma^{R}(f)(x) &= \sum_{y \in \mathcal{X}} R(x,y)\,\Gamma^{R}(f)(y) - \Gamma^{R}(f)(x) \\
&= \frac{1}{2} \sum_{y, z \in \mathcal{X}} R(x,y) R(y,z) \left(f(z) - f(y)\right)^{2} - \frac{1}{2} \sum_{y \in \mathcal{X}} R(x,y) f(y)^{2}. 
\end{align*}
Furthermore, by $f(x) = 0$,
\[ \Gamma^{R}(f, \mathcal{L}^{R} f)(x) = \frac{1}{2} \sum_{y \in \mathcal{X}} R(x,y) f(y) \left(\mathcal{L}^{R} f(y) - \mathcal{L}^{R} f(x)\right) \]
\[ = \frac{1}{2} \sum_{y, z \in \mathcal{X}} R(x,y) R(y,z)\, f(y) \left(f(z) - f(y)\right) - \frac{1}{2} \left(\sum_{y \in \mathcal{X}} R(x,y) f(y)\right)^{2}. \]
By the elementary identity
$\frac{1}{2}(a - b)^{2} - b(a - b) = \frac{1}{2}(a - 2b)^{2} - \frac{1}{2} b^{2}$ for $a, b \in \mathbb{R}$,
applied to $a = f(z)$ and $b = f(y)$, together with $\sum_{z \in \mathcal{X}} R(y,z) = 1$,
\[ \frac{1}{2} \sum_{y, z \in \mathcal{X}} R(x,y) R(y,z) \left(f(z) - f(y)\right)^{2} - \sum_{y, z \in \mathcal{X}} R(x,y) R(y,z)\, f(y) \left(f(z) - f(y)\right) \]
\[ = \frac{1}{2} \sum_{y, z} R(x,y) R(y,z) \left(f(z) - 2f(y)\right)^{2} - \frac{1}{2} \sum_{y} R(x,y) f(y)^{2}. \]
Therefore,
\[ 2\,\Gamma_{2}^{R}(f)(x) = \mathcal{L}^{R}\Gamma^{R}(f)(x) - 2\,\Gamma^{R}(f, \mathcal{L}^{R} f)(x) \]
\[ = \left(\sum_{y \in \mathcal{X}} R(x,y) f(y)\right)^{2} - \sum_{y \in \mathcal{X}} R(x,y) f(y)^{2}  + \frac{1}{2} \sum_{y, z \in \mathcal{X}} R(x,y) R(y,z) \left(f(z) - 2f(y)\right)^{2}. \qedhere \]
\end{proof}

\begin{proof}[Proof of Theorem \ref{thm:abelian}]
Let $d \coloneqq |S|$.
Every $s \in S$ has finite order, so $s^{-1}$ is a non-negative power of $s$.
Thus $S$ generates $G$ as a semigroup, and $(G,P_S)$ is irreducible.
Moreover, $\sum_{x \in G} P_S(x,y) = \frac{|yS^{-1}|}{d} = 1$ for every $y \in G$,
so the uniform distribution is stationary.

Fix $x \in G$ and $f \colon G \to \mathbb{R}$.
Since $\Gamma_2^{P_S}(f)(x)$ is unchanged by adding a constant to $f$, we can assume that $f(x)=0$.
By expanding the formula in Lemma \ref{lem:gamma2-formula}, we see that 
\begin{equation}\label{eq:abelian-expansion}
2d^2\,\Gamma_2^{P_S}(f)(x) = \left(\sum_{s \in S} f(xs)\right)^2+d\sum_{s \in S} f(xs)^2 + \frac12\sum_{s,t \in S}\left(f(xst)^2-4f(xst)f(xs)\right).
\end{equation}
Since $xst=xts$, interchanging $s$ and $t$ yields
\[ \sum_{s,t \in S} f(xst)f(xs) = \frac12\sum_{s,t \in S} f(xst)\left(f(xs)+f(xt)\right). \]
Also,
\[ \frac12 \sum_{s,t \in S}\left(f(xs)+f(xt)\right)^2 = d\sum_{s \in S} f(xs)^2+\left(\sum_{s \in S} f(xs)\right)^2.\]

By substituting these identities into \eqref{eq:abelian-expansion} and completing the square,
we obtain that 
\begin{equation}\label{eq:abelian-squares}
\Gamma_2^{P_S}(f)(x) = \frac{1}{4d^2}\sum_{s,t \in S}\left(f(xst)-f(xs)-f(xt)+f(x)\right)^2\ge0.
\end{equation}
Here the sum is over all ordered pairs $(s,t) \in S^2$.
Thus $\textup{BE}(G,P_S)\ge0$.
\end{proof}

\begin{Exa}\label{exa:z12}
We give a finite Abelian group with an asymmetric generating set.
Consider the quotient group $G = \mathbb{Z} / 12\mathbb{Z}$, which we write additively, so that $S^{-1}$ reads $-S = \left\{-s \,\middle|\, s \in S\right\}$ for $S \subset G$.
Let $S = \{1,4,7\}$ and let $P_S$ be the stochastic matrix associated with $S$.
Then $S \ne -S$, $S$ generates $G$, and $k_{(G, P_S)} = 2$. 
Indeed, $S$ generates $G$ since $1 \in S$.
Furthermore, $-1 = 4 + 7$, $-4 = 1 + 7$ and $-7 = 1 + 4$, while none of $-1$, $-4$ and $-7$ belongs to $S$.
Hence $S \ne -S$ and $d_{P_S}(0 \to -s) = 2$ for every $s \in S$, and $k_{(G, P_S)} = 2$ follows from Lemma \ref{lem:cayley-spectral} (3) below.
\end{Exa}

\subsection{A random example on $(\mathbb{Z}/3\mathbb{Z})^{d}$}\label{subsec:random}

In this subsection, we give a sequence of random walks on finite Abelian groups with \emph{non-symmetric} generating sets to which Theorem \ref{thm:main} applies almost surely.
It is a counterpart, without the symmetric support condition, of \cite[Corollary 4]{Salez2024},
where the generating sets are symmetric.
The groups are $(\mathbb{Z}/3\mathbb{Z})^{d}$, $d \ge 1$.
The reason for this choice is that every non-zero element of $(\mathbb{Z}/3\mathbb{Z})^{d}$ has order $3$,
which forces $k_{(G, P_{S})} \le 2$ for \emph{every} generating set $S$ (Lemma \ref{lem:k-exponent-3} below).
Consequently the whole generating set can be chosen at random, without any symmetrization.

Throughout this subsection, we deal with additive groups and write finite Abelian groups {\it additively}, with identity $0$.
Thus, for a generating set $S$ of a finite Abelian group $G$, the stochastic matrix $P_{S}$ of \eqref{eq:tr-pb-def} reads
\[ P_{S}(x, x+s) = \frac{1}{|S|}, \quad x \in G,\ s \in S, \]
and $S^{-1}$ reads $-S \coloneqq \left\{-s \,\middle|\, s \in S\right\}$. 

We deal with the Fourier analysis on finite groups. 
We denote by $\widehat{G}$ the group of characters of $G$, that is, of homomorphisms $\chi \colon G \to \left\{z \in \mathbb{C} \,\middle|\, |z| = 1\right\}$,
and by $1$ the trivial character.
For a non-empty subset $X$ of $G$, let
\[ \widehat{p_{X}}(\chi) \coloneqq \frac{1}{|X|} \sum_{x \in X} \chi(x), \quad \chi \in \widehat{G}, \qquad
   \lambda_{*}(X) \coloneqq \max_{\chi \in \widehat{G},\ \chi \ne 1} \mathrm{Re}\, \widehat{p_{X}}(\chi). \]

We first express the quantities appearing in \eqref{eq:main-assumption} in terms of $S$.

\begin{Lem}\label{lem:cayley-spectral}
Let $G$ be a finite Abelian group with $|G| \ge 2$ and let $S$ be a generating set of $G$.
Then the following hold for $(G, P_{S})$.\\
(1) $\Delta(P_{S}) = |S|$.\\
(2) $\lambda_{*}(S) < 1$ and $t_{\mathrm{rel}} = \left(1 - \lambda_{*}(S)\right)^{-1}$.\\
(3) $k_{(G, P_{S})} = \max\left\{ d_{P_{S}}(0 \to -s) \,\middle|\, s \in S \setminus \{0\} \right\}$.\\
(4) $k_{(G, P_{S})} = 1$ if and only if $S = -S$.
\end{Lem}

\begin{proof}
Let $N \coloneqq |G|$ and let $\pi$ be the uniform distribution on $G$, which is the stationary distribution of $(G, P_{S})$ by Theorem \ref{thm:abelian}.

(1)
For $x, y \in G$ with $x \ne y$, we have $P_{S}(x,y) > 0$ if and only if $y - x \in S$, and in this case $P_{S}(x,y) = 1/|S|$.
Since $|G| \ge 2$ and $S$ generates $G$, we see that $S \setminus \{0\} \ne \emptyset$, and hence there exist $x \ne y$ with $P_{S}(x,y) > 0$.
Therefore $\Delta(P_{S}) = |S|$.

(2)
For $f \colon G \to \mathbb{R}$, let $\widehat{f}(\chi) \coloneqq \frac{1}{N} \sum_{x \in G} f(x) \overline{\chi(x)}$, $\chi \in \widehat{G}$. 
This is the Fourier transform of $f$. 
Then the inversion formula says that $f = \sum_{\chi \in \widehat{G}} \widehat{f}(\chi) \chi$ and the Plancherel formula says that $\frac{1}{N} \sum_{x \in G} f(x)^{2} = \sum_{\chi \in \widehat{G}} |\widehat{f}(\chi)|^{2}$.
Since $\mathbb{E}_{\pi}[f] = \widehat{f}(1)$, we obtain that 
\begin{equation}\label{eq:var-char}
\mathrm{Var}_{\pi}(f) = \sum_{\chi \ne 1} |\widehat{f}(\chi)|^{2}. 
\end{equation} 

Fix $s \in S$. Since $f(x+s) - f(x) = \sum_{\chi} \widehat{f}(\chi) \left(\chi(s) - 1\right) \chi(x)$ for $x \in G$,
\[ \frac{1}{N} \sum_{x \in G} \left(f(x+s) - f(x)\right)^{2} = \sum_{\chi} |\widehat{f}(\chi)|^{2} \left|\chi(s) - 1\right|^{2}  = 2 \sum_{\chi} |\widehat{f}(\chi)|^{2} \left(1 - \mathrm{Re}\,\chi(s)\right). \]
Averaging over $s \in S$, we obtain that 
\begin{equation}\label{eq:DF-char} 
\mathcal{E}(f) = \frac{1}{2N|S|} \sum_{x \in G} \sum_{s \in S} \left(f(x+s) - f(x)\right)^{2}   = \sum_{\chi \ne 1} |\widehat{f}(\chi)|^{2} \left(1 - \mathrm{Re}\, \widehat{p_{S}}(\chi)\right). 
\end{equation}
If $\mathrm{Re}\, \widehat{p_{S}}(\chi) = 1$ for some $\chi \in \widehat{G}$, then $\chi(s) = 1$ for every $s \in S$, and hence $\chi = 1$ on the subgroup generated by $S$, which is $G$. 
Hence $\chi \equiv 1$ on $G$. 
This shows $\lambda_{*}(S) < 1$. 
By \eqref{eq:var-char} and \eqref{eq:DF-char}, 
$\mathcal{E}(f) \ge \left(1 - \lambda_{*}(S)\right) \mathrm{Var}_{\pi}(f)$ for every $f$,
and hence $t_{\mathrm{rel}} \le \left(1 - \lambda_{*}(S)\right)^{-1}$.

Conversely, let $\chi_{0} \in \widehat{G} \setminus \{1\}$ satisfy $\mathrm{Re}\, \widehat{p_{S}}(\chi_{0}) = \lambda_{*}(S)$ and let $f_{0} \coloneqq \mathrm{Re}\, \chi_{0} = \frac{1}{2}\left(\chi_{0} + \overline{\chi_{0}}\right)$.
Then $f_{0}$ is real-valued and non-constant, $\widehat{f_{0}}$ vanishes outside $\{\chi_{0}, \overline{\chi_{0}}\}$,
and $\mathrm{Re}\, \widehat{p_{S}}(\overline{\chi_{0}}) = \mathrm{Re}\, \overline{\widehat{p_{S}}(\chi_{0})} = \lambda_{*}(S)$.
By \eqref{eq:DF-char}, 
$\mathcal{E}(f_{0}) = \left(1 - \lambda_{*}(S)\right) \mathrm{Var}_{\pi}(f_{0})$. 
Since $\chi_0 \ne 1$, $\mathrm{Var}_{\pi}(f_{0}) > 0$. 
By \eqref{eq:def-relax}, $t_{\mathrm{rel}} \ge \left(1 - \lambda_{*}(S)\right)^{-1}$.

(3)
Since $P_{S}(x+g, y+g) = P_{S}(x,y)$ for every $x, y, g \in G$, we have $P_{S}^{k}(x+g, y+g) = P_{S}^{k}(x,y)$ for every $k \in \mathbb{N}_{0}$ by induction,
and hence $d_{P_{S}}(x \to y) = d_{P_{S}}(0 \to y - x)$.
By the definition of $d_{P_{S}}$, we have $d_{P_{S}}(x \to y) = 1$ if and only if $y \ne x$ and $P_{S}(x,y) > 0$, that is, $y - x \in S \setminus \{0\}$.
For such $x, y$, putting $s \coloneqq y - x$, we obtain $d_{P_{S}}(y \to x) = d_{P_{S}}(0 \to x - y) = d_{P_{S}}(0 \to -s)$.
Taking the maximum over all such pairs yields the assertion.

(4)
Assume $S = -S$. 
For $s \in S \setminus \{0\}$, we have $-s \in S \setminus \{0\}$, and hence $P_{S}(0, -s) > 0$ and $-s \ne 0$, that is, $d_{P_{S}}(0 \to -s) = 1$.
By (3), $k_{(G, P_{S})} = 1$.

Conversely, assume $k_{(G, P_{S})} = 1$. 
By (3), $d_{P_{S}}(0 \to -s) = 1$ for every $s \in S \setminus \{0\}$, and hence $-s \in S$.
Since $-S \subset S$, and $|-S| = |S|$, we see that $S = -S$.
\end{proof}

\begin{Lem}\label{lem:k-exponent-3}
Let $d \ge 1$, $G \coloneqq (\mathbb{Z}/3\mathbb{Z})^{d}$, and let $S$ be a generating set of $G$.
Then $k_{(G, P_{S})} \le 2$, and $k_{(G, P_{S})} = 2$ if and only if $S \ne -S$.
\end{Lem}

\begin{proof}
Let $s \in S \setminus \{0\}$. 
Since $3s = 0$, 
\[ P_{S}^{2}(0, -s) \ge P_{S}(0, s)\, P_{S}(s, s+s) = \frac{1}{|S|^{2}} > 0, \]
that is, $d_{P_{S}}(0 \to -s) \le 2$. 
By Lemma \ref{lem:cayley-spectral} (3), $k_{(G, P_{S})} \le 2$.
The second assertion follows from Lemma \ref{lem:cayley-spectral} (4) and $k_{(G, P_{S})} \ge 1$.
\end{proof}

We now consider random generating sets. The following lemma is a quantitative form of the fact that random Cayley graphs of Abelian groups are expanders \cite{AR1994}.

\begin{Lem}\label{lem:random-subset}
Let $G$ be a finite Abelian group with $N \coloneqq |G| \ge 2$, let $1 \le m \le N-1$,
and let $X$ be a random subset of $G$ which is uniformly distributed over the $m$-element subsets of $G \setminus \{0\}$.
Then, \\
(1) For every $\rho \in (0,1)$,
\[ \mathbb{P}\left(\lambda_{*}(X) \ge \rho\right) \le (N-1) \exp\left(-\frac{m \rho^{2}}{2}\right). \]
(2) If $-g \ne g$ for every $g \in G \setminus \{0\}$, then
\[ \mathbb{P}\left(X = -X\right) \le \frac{m-1}{N-2}. \]
\end{Lem}

\begin{proof}
(1)
Fix $\chi \in \widehat{G}$ with $\chi \ne 1$. 
The numbers $\mathrm{Re}\, \chi(x)$, $x \in G \setminus \{0\}$, lie in $[-1,1]$. 
Since $\chi \ne 1$, there exists $a \in G$ such that $\chi(a) \ne 1$. 
Then, $\sum_{x \in G} \chi(x) = \sum_{x \in G} \chi(a+x) = \chi(a) \sum_{x \in G} \chi(x)$, and hence $\sum_{x \in G} \chi(x) = 0$. 
Therefore, 
\[ \mu_{\chi} \coloneqq \frac{1}{N-1} \sum_{x \in G \setminus \{0\}} \mathrm{Re}\, \chi(x) = -\frac{1}{N-1} < 0. \]
The random variable $m\, \mathrm{Re}\, \widehat{p_{X}}(\chi) = \sum_{x \in X} \mathrm{Re}\, \chi(x)$ is the sum of a sample of size $m$ drawn without replacement from this population.
Hoeffding's inequality \cite[Theorem 2]{Hoeffding1963} for sums of independent random variables with values in $[-1,1]$,
which remains valid for sampling without replacement by \cite[Section 6]{Hoeffding1963}, yields that for every $t > 0$,
\[ \mathbb{P}\left( \sum_{x \in X} \mathrm{Re}\, \chi(x) - m \mu_{\chi} \ge t \right) \le \exp\left(-\frac{t^{2}}{2m}\right). \]
Taking $t \coloneqq m\left(\rho - \mu_{\chi}\right) \ge m\rho$, we obtain that 
\[ \mathbb{P}\left( \mathrm{Re}\, \widehat{p_{X}}(\chi) \ge \rho \right) \le \exp\left(-\frac{m\left(\rho - \mu_{\chi}\right)^{2}}{2}\right) \le \exp\left(-\frac{m\rho^{2}}{2}\right). \]
Since $\lambda_{*}(X) \ge \rho$ if and only if $\mathrm{Re}\, \widehat{p_{X}}(\chi) \ge \rho$ for some $\chi \ne 1$,
the assertion follows by summing over the $N-1$ non-trivial characters.

(2)
Let $(X_{1}, \dots, X_{m})$ be uniformly distributed over the $m$-tuples of distinct elements of $G \setminus \{0\}$.
Then $\{X_{1}, \dots, X_{m}\}$ is uniformly distributed over the $m$-element subsets of $G \setminus \{0\}$, so we can assume that $X = \{X_{1}, \dots, X_{m}\}$.
If $X = -X$, then $-X_{1} \in X$, and $-X_{1} \ne X_{1}$ by assumption, so that $-X_{1} \in \{X_{2}, \dots, X_{m}\}$.
Conditionally on $X_{1}$, the tuple $(X_{2}, \dots, X_{m})$ is uniformly distributed over the $(m-1)$-tuples of distinct elements of $G \setminus \{0, X_{1}\}$,
a set of $N-2$ elements containing $-X_{1}$, and hence the conditional probability that $-X_{1} \in \{X_{2}, \dots, X_{m}\}$ equals $(m-1)/(N-2)$.
\end{proof}

\begin{Thm}\label{thm:z3d-random}
Let $c > 2 \log 3$, let $\rho \in \left(\sqrt{\frac{2 \log 3}{c}},\, 1\right)$, and let $\gamma \coloneqq \frac{c \rho^{2}}{2} - \log 3 > 0$.
For $d \ge 1$, let $G_{d} \coloneqq (\mathbb{Z}/3\mathbb{Z})^{d}$ and $m_{d} \coloneqq \lceil c d \rceil$,
and let $d_{0} \ge 1$ satisfy $m_{d} \le 3^{d} - 1$ for every $d \ge d_{0}$.
Let $(S_{d})_{d \ge d_{0}}$ be random subsets defined on a common probability space such that, for every $d \ge d_{0}$,
$S_{d}$ is uniformly distributed over the $m_{d}$-element subsets of $G_{d} \setminus \{0\}$.
Then the following hold.\\
(1) For every $d \ge d_{0}$, with probability at least $1 - e^{-\gamma d} - m_{d}\, 3^{1-d}$,
the set $S_{d}$ generates $G_{d}$, $S_{d} \ne -S_{d}$, and the finite irreducible Markov chain $(G_{d}, P_{S_{d}})$ satisfies
\[ k_{(G_{d}, P_{S_{d}})} = 2, \qquad \Delta(P_{S_{d}}) = m_{d}, \qquad t_{\mathrm{rel}} \le \frac{1}{1 - \rho}, \qquad \textup{BE}(G_{d}, P_{S_{d}}) \ge 0. \]
(2) Almost surely, there exists $d_{1} \ge d_{0}$ such that the conclusion of (1) holds for every $d \ge d_{1}$,
the sequence $\left(G_{d}, P_{S_{d}}\right)_{d \ge d_{1}}$ satisfies \eqref{eq:main-assumption}, and a cutoff occurs for it.
\end{Thm}

\begin{proof}
(1) 
Fix $d \ge d_{0}$ and let $N_{d} \coloneqq |G_{d}| = 3^{d}$.
By Lemma \ref{lem:random-subset} (1) with $m = m_{d} \ge cd$,
\[ \mathbb{P}\left(\lambda_{*}(S_{d}) \ge \rho\right) \le 3^{d} \exp\left(-\frac{c d \rho^{2}}{2}\right)  = e^{-\gamma d}. \]
Every non-zero element of $G_{d}$ has order $3$, so $-g \ne g$ for $g \ne 0$, and Lemma \ref{lem:random-subset} (2) gives
\[ \mathbb{P}\left(S_{d} = -S_{d}\right) \le \frac{m_{d} - 1}{3^{d} - 2} \le m_{d}\, 3^{1-d}. \]

Consider the event 
\[ \mathcal{A}_d \coloneqq \left\{\lambda_{*}(S_{d}) < \rho\right\} \cap \left\{S_{d} \ne -S_{d}\right\},\] 
whose probability is at least $1 - e^{-\gamma d} - m_{d} 3^{1-d}$.

On $\mathcal{A}_d$, the set $S_{d}$ generates $G_{d}$. 
Otherwise the subgroup $H$ generated by $S_{d}$ is proper. 
Then there exists a non-trivial character of $G/H$ and 
hence there exists $\chi \in \widehat{G_{d}}$ with $\chi \ne 1$ and $\chi \equiv 1$ on $H$. 
Now we see that 
\[ \lambda_{*}(S_{d}) = \mathrm{Re}\, \widehat{p_{S_{d}}}(\chi) = 1 > \rho, \]
which is a contradiction.

By Theorem \ref{thm:abelian}, $(G_{d}, P_{S_{d}})$ is a finite irreducible Markov chain with $\textup{BE}(G_{d}, P_{S_{d}}) \ge 0$.
By Lemma \ref{lem:cayley-spectral} (1) and (2), $\Delta(P_{S_{d}}) = m_{d}$ and $t_{\mathrm{rel}} = \left(1 - \lambda_{*}(S_{d})\right)^{-1} \le (1-\rho)^{-1}$.
By Lemma \ref{lem:k-exponent-3} and $S_{d} \ne -S_{d}$, $k_{(G_{d}, P_{S_{d}})} = 2$.

(2)
Since $\sum_{d \ge d_{0}} \left(e^{-\gamma d} + m_{d}\, 3^{1-d}\right) < \infty$, the Borel--Cantelli lemma implies that almost surely $\mathcal{A}_d$ fails for only finitely many $d$.
Fix such a realization and let $d_{1} \ge d_{0}$ be such that $\mathcal{A}_d$ holds for every $d \ge d_{1}$.
Then $\left(G_{d}, P_{S_{d}}\right)_{d \ge d_{1}}$ is a sequence of finite irreducible Markov chains with non-negative Bakry--\'Emery curvature. 
In the notation of Proposition \ref{prop:sufficient}, we see that $k_{d} = 2$, $\Delta_{d} = m_{d} \le cd + 1$, $t_{\mathrm{rel}}^{(d)} \le (1-\rho)^{-1}$ and $|G_{d}| = 3^{d}$. 
Therefore, 
\[ t_{\mathrm{rel}}^{(d)} \sqrt{\frac{(1+k_{d})^{5}\left(1+\log \Delta_{d}\right)^{3}}{\log |G_{d}|}}   \le \frac{1}{1-\rho} \sqrt{\frac{3^{5} \left(1 + \log(cd+1)\right)^{3}}{d \log 3}}  \to 0, \quad d \to \infty. \]
Hence \eqref{eq:suff} holds, and Proposition \ref{prop:sufficient} yields \eqref{eq:main-assumption} and the cutoff.
\end{proof}

By Theorem \ref{thm:discrete}, we also obtain a cutoff in discrete time for the lazy versions of these chains.

\begin{Cor}\label{cor:z3d-lazy}
Let $\delta \in (0,1)$, and let $c, \rho, d_{0}, (S_{d})_{d \ge d_{0}}$ be as in Theorem \ref{thm:z3d-random}.
For $d \ge d_{0}$, let $P'_{d} \coloneqq \delta I + (1 - \delta) P_{S_{d}}$.
Then, almost surely, there exists $d_{1} \ge d_{0}$ such that $\left(G_{d}, P'_{d}\right)_{d \ge d_{1}}$ is a sequence of finite irreducible $\delta$-lazy Markov chains
with non-negative Bakry--\'Emery curvature satisfying \eqref{eq:main-assumption},
and a cutoff occurs for its discrete mixing times.
\end{Cor}

\begin{proof}
Let $d_{1}$ be as in Theorem \ref{thm:z3d-random} (2), and let $d \ge d_{1}$.
The chain $(G_{d}, P'_{d})$ is $\delta$-lazy, and $P_{S_{d}} = (1-\delta)^{-1}\left(P'_{d} - \delta I\right)$.
Since $(P'_{d})^{k} \ge (1-\delta)^{k} P_{S_{d}}^{k}$ entrywise for every $k \in \mathbb{N}_{0}$, the chain $(G_{d}, P'_{d})$ is irreducible, and its stationary distribution is the uniform distribution.
Hence Lemma \ref{lem:lazy-comparison}, applied to $P = P'_{d}$ and $Q = P_{S_{d}}$, shows that
\[ k_{(G_{d}, P'_{d})} = k_{(G_{d}, P_{S_{d}})}, \qquad \Delta(P'_{d}) = \frac{\Delta(P_{S_{d}})}{1 - \delta}, \qquad
   t_{\mathrm{rel}}(P'_{d}) = \frac{t_{\mathrm{rel}}(P_{S_{d}})}{1 - \delta}, \]
\[ \textup{BE}(G_{d}, P'_{d}) = (1 - \delta)\, \textup{BE}(G_{d}, P_{S_{d}}) \ge 0, \qquad
   t_{\mathrm{mix}}^{P'_{d}}(\eta) = \frac{t_{\mathrm{mix}}^{P_{S_{d}}}(\eta)}{1 - \delta}, \quad \eta \in (0,1), \]
where $t_{\mathrm{rel}}(P'_{d})$ and $t_{\mathrm{rel}}(P_{S_{d}})$ denote the relaxation times, and $t_{\mathrm{mix}}^{P'_{d}}$ and $t_{\mathrm{mix}}^{P_{S_{d}}}$ the mixing times, of the two chains.
Since
\[ 1 + \log \Delta(P'_{d}) \le \left(1 + \log \frac{1}{1-\delta}\right)\left(1 + \log \Delta(P_{S_{d}})\right), \]
the quantity in \eqref{eq:main-assumption} for $(G_{d}, P'_{d})$ is at most $\left(1 - \delta\right)^{-1/2} \left(1 + \log \frac{1}{1-\delta}\right)$ times the one for $(G_{d}, P_{S_{d}})$,
which tends to $0$ by Theorem \ref{thm:z3d-random} (2). The cutoff in discrete time follows from Theorem \ref{thm:discrete}.
\end{proof}

\noindent{\it Acknowledgements} \ The author has received funding from JSPS KAKENHI Grant Number JP22K13928.

\bibliographystyle{plain}
\bibliography{cutoff_asym}

\end{document}